\documentclass[10pt,a4paper]{amsart}
\usepackage{verbatim}

\usepackage[toc]{appendix}
\usepackage[T1]{fontenc}
\usepackage{graphicx}
\usepackage{enumerate}
\usepackage{amsmath,amsfonts,amssymb}
\usepackage{color}
\def\loc{\operatorname{loc}}
\usepackage{cite}
\usepackage{ latexsym }
\definecolor{citation}{rgb}{0.11,0.67,0.84}
\definecolor{formula}{rgb}{0.1,0.2,0.6}
\definecolor{url}{rgb}{0.11,0.67,0.84}
\usepackage{pgf,tikz}
\usepackage{mathrsfs}
\usepackage{fancyhdr}
\usepackage{dutchcal}

\newcommand{\medint}{-\kern -,375cm\int}

\newcommand{\medintinrigo}{-\kern -,315cm\int}
\makeatletter
\newcommand{\linethrough}{\mathpalette\@thickbar}
\newcommand{\@thickbar}[2]{{#1\mkern0mu\vbox{
    \sbox\z@{$#1#2\mkern-0.5mu$}%
    \dimen@=\dimexpr\ht\tw@-\ht\z@+2\p@\relax % The +2 represents the vertical shift of the line.
    \hrule\@height0.5\p@ % The 0.5 represent the thickness on the line.
    \vskip\dimen@
    \box\z@}}
}
\makeatother

\newtheorem{theorem}{Theorem}[section]
\newtheorem{lemma}[theorem]{Lemma}
\newtheorem{proposition}[theorem]{Proposition}

\newtheorem{remark}[theorem]{Remark}
\numberwithin{equation}{section}

\usepackage{dsfont}

\newcommand{\reqnomode}{\tagsleft@false}

\usepackage{hyperref}

\def\dx{\,{\rm d}x}

\def\dy{\,{\rm d}y}

\def \d{\,{\rm d}}
\def \diver{\,{\rm div}}
\def\dist{\,{\rm dist}}

\allowdisplaybreaks
\makeatletter
\DeclareRobustCommand*{\bfseries}{%
  \not@math@alphabet\bfseries\mathbf
  \fontseries\bfdefault\selectfont
  \boldmath
}

\makeatother

\newlength{\defbaselineskip}
\newcommand{\mint}{\mathop{\int\hskip -1,05em -\, \!\!\!}\nolimits}

\def \diver{\,{\rm div}}

\newcommand{\ti}[1]{\tilde{#1}}

\newcommand{\rr}{\varrho}
\newcommand{\snr}[1]{\lvert #1\rvert}
\newcommand{\nr}[1]{\lVert #1 \rVert}

\newcommand{\tx}[1]{\textnormal{\texttt{#1}}}

\def\loc{\operatorname{loc}}
\def\eqn#1$$#2$${\begin{equation}\label#1#2\end{equation}}

\newcommand{\jj}{\mathcal{j}}

\def\XXint#1#2#3{{\setbox0=\hbox{$#1{#2#3}{\int}$}
     \vcenter{\hbox{$#2#3$}}\kern-.5\wd0}}

\author[De Filippis]{Cristiana De Filippis}  \address{Cristiana De Filippis\\Dipartimento SMFI, Universit\`a di Parma\\ Parco Area delle Scienze 53/A, 43124 Parma, Italy} \email{\url{cristiana.defilippis@unipr.it}}

\begin{document}

\subjclass[2020]{35B65, 35J70, 35J75, 35J93\vspace{1mm}} %%ALERT CHECK 35J60 23J70 35B65 35D40

\keywords{Regularity, Nonuniform Ellipticity, Degenerate PDEs\vspace{1mm}}

\title[Necessity of Ladyzhenskaya \& Ural'tseva's sufficient conditions]{Necessity of Ladyzhenskaya \& Ural'tseva's sufficient conditions in nonuniform ellipticity}

\thanks{{\it Acknowledgements.}\ This work is supported by the European Research Council, through the ERC StG project NEW, nr.~101220121.
\vspace{1mm}}

\begin{abstract}
The sufficient conditions of Ladyzhenskaya \& Ural'tseva \cite{lu70} for interior gradient estimates in nonuniformly elliptic problems are found to be also necessary. Specifically, nonparametric minimal surfaces exhibit the maximal nonuniformity rate compatible with the local Lipschitz regularity of solutions, and the classical theory of Finn \cite{fin54} and Bombieri \& De Giorgi \& Miranda \cite{bdm69} doesn't extend beyond Bernstein genre $2$.
 \end{abstract}
\maketitle

\setcounter{tocdepth}{1}
{\small \tableofcontents}
\section{Introduction}
\noindent In 1970, Ladyzhenskaya \& Ural'tseva observed that it seemed impossible to formulate simple, universal structural conditions encompassing the full range of nonuniformly elliptic models for which pointwise interior gradient estimates might hold \cite[p.~686-687]{lu70}. This paper settles this issue. The answer is encoded in the classical family of nonparametric area-type integrals: 
\eqn{0.1}
$$
W^{1,1}(\Omega)\ni w\mapsto \mathcal{A}_{p}(w;\Omega):=\int_{\Omega}(1+\snr{Dw}^{p})^{\frac{1}{p}}\dx,
\qquad 1<p<\infty,
$$
whose Euler-Lagrange equations are
\eqn{eqn.1}
$$
-\diver\left(\frac{\snr{Du}^{p-2}Du}{(1+\snr{Du}^{p})^{1-1/p}}\right)=0
\qquad \mbox{in}\ \ \Omega.
$$
The sharpness is surgical. For every $p>2$, integrand
\eqn{0.97}
$$
\mathbb{R}^{n}\ni z\mapsto A_{p}(z):=(1+\snr{z}^{p})^{\frac{1}{p}},
\qquad 1<p<\infty.
$$
retains every structural hypothesis of \cite[Theorem~1]{lu70} except the decisive ellipticity/nonuniformity balance condition \cite[(1.12)]{lu70}, and the conclusion fails. More precisely, for every $p>2$ - and therefore for arbitrarily small superquadratic departures from the minimal-surface $p=2$ regime - we construct, in any dimension $n\ge 2$, a bounded domain with smooth boundary and a $C^{\infty}$-Dirichlet datum for which the autonomous and strictly convex functional $\mathcal{A}_{p}$ admits a unique minimizer. This minimum is bounded and Hölder-continuous, yet its gradient is unbounded in the interior, with set of non-Lebesgue points of positive $\mathcal{H}^{n-1}$-measure. Thus the Ladyzhenskaya-Ural'tseva balance between ellipticity and nonuniformity is not a limitation of their method, but a sharp structural threshold for universal interior gradient regularity in the polynomial eigenvalue scale. Our main result is indeed the following.
\begin{theorem}\label{t.2}
Let $n\ge 2$ and $2<p<\infty$. There exists a bounded domain $\Omega\subset \mathbb{R}^{n}$ with smooth boundary, and a function $v_{0}\in C^{\infty}(\bar{\Omega})$ such that the unique solution $v\in v_{0}+W^{1,1}_{0}(\Omega)$ to Dirichlet problem
\eqn{pd}
$$
v_{0}+W^{1,1}_{0}(\Omega)\ni v\mapsto \min_{w\in v_{0}+W^{1,1}_{0}(\Omega)}\mathcal{A}_{p}(w;\Omega),
$$
and unique generalized minimizer of $\mathcal{A}_{p}$ in class $v_{0}+W^{1,1}_{0}(\Omega)$, satisfies $v\in C^{0,1-\frac{2}{p}}(\bar{\Omega})\cap W^{1;\frac{p}{2},\infty}(\Omega)\setminus W^{1,\infty}_{\loc}(\Omega)$. In particular, the set of non-Lebesgue points of $Dv$ has positive $\mathcal{H}^{n-1}$-measure.
\end{theorem}
\noindent We refer to Section \ref{pre} for notation and terminology. The large-gradient ellipticity of equation \eqref{eqn.1} can be read through two complementary quantities. The first is Bernstein genre \cite{ber12,ler39}, which measures radial ellipticity - that is, the ellipticity retained in the gradient direction - relative to the total ellipticity represented by the trace. In Bernstein's terminology, the equation associated with $A_{p}$ has genre $p$, since
\eqn{i.1}
$$
\frac{\langle\partial^{2}A_{p}(z)z,z\rangle}{\textnormal{tr}(\partial^{2}A_{p}(z))}
\approx_{n,p}\snr{z}^{2-p},\qquad \quad z\in \mathbb{R}^{n}\setminus B_{1}(0).
$$
The second is the ellipticity ratio \cite{gt77}, the natural quantifier of nonuniformity, defined\footnote{Definitions \eqref{i.1}-\eqref{i.2} extend to elliptic equations $-\diver(a(Du))=0$ by replacing $\partial^{2}A_{p}$ with $\partial a$ in \eqref{i.1}-\eqref{i.2} - assume that $\partial a$ is symmetric for definiteness.} here by
\eqn{i.2}
$$
\tx{R}_{A_{p}}(z):=
\frac{\mbox{highest eigenvalue of }\partial^{2}A_{p}(z)}
{\mbox{lowest eigenvalue of }\partial^{2}A_{p}(z)}
\approx_{n,p}1+\snr{z}^{p},\qquad \quad z\in \mathbb{R}^{n}\setminus B_{1}(0).
$$
Thus $A_{p}$ is nonuniformly elliptic, with ellipticity ratio blowing up as the $p$-power of the gradient variable. Formulas \eqref{i.1}-\eqref{i.2} also explain why local Lipschitz estimates are the focal point of the theory: the local boundedness of $Du$ controls both large-gradient machineries. More importantly, they isolate the transition at $p=2$. For $1<p<2$, the quantity in \eqref{i.1} grows at infinity, for $p=2$, it stays bounded away from zero, for $p>2$, it decays to zero, while the blow-up rate in \eqref{i.2} increases dramatically. As verified in Section \ref{luc.s}, this change makes precisely \cite[(1.12)]{lu70} fail while every other relevant structural hypothesis of \cite[Theorem 1]{lu70} remains valid. Theorem \ref{t.2} shows that the interior gradient estimate fails together with this single condition. Let us give context. The modern theory of interior gradient bounds for nonuniformly elliptic equations has its classical origin in Finn's analysis of equations of minimal-surface type in two independent variables \cite{fin54,fin63} and in the work of Bombieri \& De Giorgi \& Miranda \cite{bdm69} in arbitrary dimension for the minimal-surface equation
\eqn{i.0}
$$
-\diver\left(\frac{Du}{\sqrt{1+\snr{Du}^{2}}}\right)=0,
$$
that is \eqref{eqn.1} with $p=2$. The resulting estimates give pointwise interior control of $Du$ in terms of the local oscillation of $u$, despite the strong nonuniformity of $A_{2}$. Indeed, the tangential and radial eigenvalues of $\partial^{2}A_{2}(z)$ are $(1+\snr{z}^{2})^{-1/2}$ and $(1+\snr{z}^{2})^{-3/2}$, respectively. Therefore the ellipticity ratio \eqref{i.2} has quadratic blow-up, whereas the ellipticity retained along gradients \eqref{i.1}, stays strictly positive in the large. The regularity theory for nonuniformly elliptic PDEs developed remarkably in the subsequent years; see Oskolkov \cite{osk66}, Ivo\v{c}kina \& Oskolkov \cite{io67}, Trudinger \cite{tru67,tru72}, Serrin \cite{ser69}, and Ladyzhenskaya \& Ural'tseva \cite{lu68,lu70}. Most notably, Ladyzhenskaya \& Ural'tseva \cite{lu70} identified the precise interplay between ellipticity and nonuniformity leading to local pointwise gradient estimates for general nonuniformly elliptic equations of the form
\eqn{eqn}
$$
\sum_{i,j=1}^{n}a^{i,j}(x,u,Du)\partial^{2}_{i,j}u=a(x,u,Du),
$$
with \eqref{eqn.1} as a main model. The structure of their results is particularly explicit. Part I treats large classes of equations such as \eqref{eqn} with Bernstein genre strictly smaller than $2$ - that is, $1<p<2$ in \eqref{eqn.1}\footnote{See Beck \& Schmidt \cite{bs15} for the vectorial counterpart of these results.} - whereas Part II reaches the limiting mean-curvature regime extending \eqref{i.0}. The availability of interior gradient bounds is tied to a quantitative condition balancing nonuniformity through ellipticity. For polynomial eigenvalue scales, the Part~I condition becomes a strictly subquadratic restriction on the nonuniformity rate; the limiting quadratic case is attained in \cite[Part II]{lu70} by building on the approach of \cite{bdm69}. The key aspect is that, when $p>2$, model \eqref{eqn.1} continues to satisfy every relevant structural hypothesis of \cite[Theorem~1]{lu70} except condition \cite[(1.12)]{lu70}; see Section \ref{luc.s}. Theorem \ref{t.2} then shows that the interior gradient estimate fails. The question raised by Ladyzhenskaya \& Ural'tseva \cite[p.~686-687]{lu70} repeatedly resurfaced in the literature: see Giaquinta \& Modica \& Sou\v{c}ek \cite[p.~157]{gms79b}, Ivanov \cite[p.~11-14]{iva84}, and, more recently, Bildhauer \cite[p.~119-122]{bil02}, Mingione \cite[Section~6.2]{min06}, Beck \& Schmidt \cite[p.~121]{bs13} and \cite[p.~87]{bs15}, and Schmidt \cite[p.~215]{sch14}. While \cite{bdm69,lu70} provide the first positive regularity results, to the best of our knowledge the only negative ones in autonomous equations (at the gradient level), build on Finn's example \cite{fin65}. In \cite{fin65}, locally Lipschitz solutions to Dirichlet problem \eqref{pd}$_{p=2}$, which fail the attainment of boundary traces on nonconvex domains are constructed, see also \cite{giu84} and \cite{bbm18} for the general case $p\ge 2$. In contrast, global methods lead to very strong conclusions for minima of $\mathcal{A}_{p}$: barriers on uniformly convex domains with smooth boundary data yield globally bounded gradients for every $p\in(1,\infty)$; see \cite[Theorem 2]{lu70}, \cite{ser69} and \cite{tau78}. Partial regularity also holds\footnote{Partial regularity holds without restriction on $p$ because the estimates are performed in neighborhoods of points with small excess and bounded average at a fixed scale. The latter condition mimics the availability of an a priori local Lipschitz bound. Gradient H\"older continuity then follows by observing that the degeneracy or singularity of $A_{p}$ at the origin is analogous to that of the $p$-Laplacian, so techniques designed for systems of $p$-Laplacian type can be adapted.} for minimizers of \eqref{0.1} for every $p\in(1,\infty)$ \cite{ag88,sch14}, and Theorem \ref{t.2} supplies a lower bound for the Hausdorff dimension of the singular set. This trichotomy - nonattainment of boundary traces, conditional global regularity, and unrestricted partial regularity - further emphasizes the decisive role of ellipticity and nonuniformity in the validity or failure of interior gradient estimates. Since the '70s, the theory has been advanced by many remarkable contributions; see, for an incomplete list, Ivanov \cite{iva72,iva84}, Simon \cite{sim76}, Giusti \cite{giu78,giu84}, Giaquinta \& Modica \& Sou\v{c}ek \cite{gms79a,gms79b}, Ural'tseva \& Urdaletova \cite{uu84}, Marcellini \cite{mar86,mar89,mar91}, and Zhikov \& Kozlov \& Ole\c{i}nik \cite{jko94}. In particular, the innovative approaches of Marcellini and Zhikov opened the way to a broad class of nonuniformly elliptic PDEs and functionals with polynomial nonuniformity, quantified by controlling the ellipticity ratio at infinity by a positive power of the gradient variable. This includes both superlinear $(\textnormal{p},\textnormal{q})$-nonuniformity, and therefore quantitatively superlinear polynomial-type models, and $\mu$-ellipticity, the latter being the natural linear-growth extension of \eqref{0.1}.\footnote{Integrand $A_{p}$ in \eqref{0.97} is $\mu$-elliptic with $\mu=p+1$, see \cite[Section 1]{bs15}.} To describe both regimes, let us consider general autonomous functionals
$$
W^{1,1}(\Omega)\ni w\mapsto\int_{\Omega}F(Dw)\dx
$$
whose polynomial eigenvalue scale is prescribed by\footnote{Conditions \eqref{i.3} are understood in the sense of matrices. The same conditions apply to elliptic equations $-\diver(a(Du))=0$ after replacing $\partial^{2}F$ by $\partial a$ in \eqref{i.3}.}
\eqn{i.3}
$$
\begin{array}{c}
\displaystyle
\snr{z}^{\textnormal{p}-2}\mathds{I}_{n}\lesssim\partial^{2}F(z)
\lesssim\snr{z}^{\textnormal{q}-2}\mathds{I}_{n},
\qquad z\in\mathbb{R}^{n}\setminus B_{1}(0),\vspace{1.5mm}\\
\displaystyle
-\infty<\textnormal{p}\le \textnormal{q}<\infty,
\end{array}
$$
see e.g. \cite[Chapter 3]{bil03}. The classical subrange $-1<\textnormal{p}\le\textnormal{q}<\infty$ unifies superlinear $(\textnormal{p},\textnormal{q})$-nonuniformity, obtained when $1<\textnormal{p}\le \textnormal{q}<\infty$, and $\mu$-ellipticity with $1<\mu=2-\textnormal{p}<3$ and $\textnormal{q}=1$. Conditions \eqref{i.3} imply
\eqn{i.4}
$$
\tx{R}_{F}(z)\lesssim\snr{z}^{\textnormal{q}-\textnormal{p}}
\qquad \mbox{for all} \ \ z\in\mathbb{R}^{n}\setminus B_{1}(0).
$$
Since the work of Marcellini and Zhikov, integrals governed by polynomial ellipticity such as \eqref{i.3} have been intensively investigated: see \cite{elm04,bs13,ckp14,bs15,bs20,hs21,dkk24,sch24,bgs26,fps26} for zero- and first-order local regularity, \cite{bm20,dm21,bs24} for nonlinear potential theory, \cite{dm23,dm25,ddp26} for Schauder theory, \cite{sch09,sch14,gk19,gme21,def22,gk24} for partial regularity, and \cite{min06} for a survey. In all local regularity results, a quantitative restriction on the gap $\textnormal{q}-\textnormal{p}$, and hence on the blow-up rate in \eqref{i.4}, plays a central role. The admissible gap depends on the a priori information available on minimizers: without an a priori ansatz, a dimension-dependent smallness condition such as $\textnormal{q}-\textnormal{p}<\tx{O}(n^{-1})$ is required \cite{mar89,mar91,bil02,bs20}. In view of the classical local $L^{\infty}$-Lipschitz estimates in \cite{bdm69,lu70}, and of the boundedness supplied either by zero-order De Giorgi theory \cite{mar91,hs21} or by the maximum principle for Dirichlet problems, the natural class to consider here is that of bounded minimizers. Ladyzhenskaya \& Ural'tseva \cite{lu70} had already identified quadratic nonuniformity as the likely limiting regime. Within the polynomial eigenvalue scale, the condition $\textnormal{q}-\textnormal{p}<2$ is sufficient for interior gradient estimates \cite{cho92,elm99,bil03,ckp11}. Our Theorem \ref{t.2} grants necessity, and therefore, sharpness. 
\subsection{Nonuniform ellipticity and regularity failures}\label{ss.ss} In \cite[p.~687]{lu70} Ladyzhenskaya \& Ural'tseva conjectured the impossibility of a universal structural criterion for pointwise interior gradient estimates in nonuniformly ellitpic PDEs. Giaquinta \& Modica \& Souček \cite[Example 3.2]{gms79b} identified this issue in the lack of compensation between ellipticity and nonuniformity with a very simple example. They constructed a generalized solution $u\in BV(-1,1)$ to problem
\eqn{i.6}
$$
BV(-1,1)\ni u\mapsto \begin{cases}
    \displaystyle
    \ \inf_{w\in W^{1,1}(-1,1)}\int_{-1}^{1}\left(1+(1+x^{2})\snr{w'}^{p}\right)^{1/p}\dx,\qquad p>2\vspace{1.5mm}\\
    \displaystyle
    \ w(-1)=-m,\qquad \quad w(1)=m,
\end{cases}
$$
where $m\in (0,\infty)$ is a constant satisfying
$$
\int_{-1}^{1}\left((1+x^{2})^{\frac{p}{p-1}}-1\right)^{-\frac{1}{p}}\dx<m.
$$
Notice that the integral on the left is finite precisely when $p>2$. It turns out that $u$ is smooth in $(-1,1)\setminus \{0\}$, attains the boundary traces, but has a jump in zero, so that\footnote{Computations analogous to those in Section \ref{luc.s} show that the integrand in \eqref{i.6} satisfies all assumptions of \cite[Theorem 1]{lu70} but two, i.e. \cite[(1.12) and (1.30)]{lu70}. More precisely, the failure of (1.12) is due to the too large nonuniformity rate, while (1.30) doesn't hold because of the specific choice of the coefficient.} $u\in BV(-1,1)\setminus W^{1,1}_{\loc}(-1,1)$. Although the density in \eqref{i.6} looks like an elementary nonautonomous counterpart of $A_{p}$, the apparently harmless coefficient $1+x^{2}$ is precisely what generates and localizes the interior singularity, see \cite[Section 3]{gms79b} and \cite[Section 4.4]{bil03} for a two-dimensional extension. The fact that $x\mapsto (1+x^{2})$ has a strict minimum at $x=0$, makes it more convenient for minima to jump in zero rather than at the endpoints, thus making $x=0$ a sort of "artificial boundary", and transferring in the interior Finn's loss of boundary trace \cite{fin65}. This does not happen by chance. Since the foundational works of Lavrentiev \cite{lav27} and Manià \cite{man34} coefficients have been known to affect dramatically both relaxation and regularity - their effect can be devastating already in uniformly elliptic models. This mechanism acquired a systematic form in the work of Zhikov \cite{zhi95,zhi97}, who introduced a series of basic functionals such as
\eqn{i.7}
$$
\begin{array}{c}
\displaystyle
W^{1,\textnormal{p}}(Q_{1}(0))\ni w\mapsto \mathcal{P}(w;Q_{1}(0)):=\int_{Q_{1}(0)}\snr{Dw}^{\textnormal{p}}+a(x)\snr{Dw}^{\textnormal{q}}\dx\vspace{1.5mm}\\
\displaystyle
1<\textnormal{p}<\textnormal{q}<\infty,\qquad \quad 0\le a(\cdot)\in C^{0,1}(Q_{1}(0)),
\end{array}
$$
and showed the occurrence of Lavrentiev phenomenon, that is
$$\inf_{w\in u_{0}+W^{1,\textnormal{p}}_{0}(Q_{1}(0))}\mathcal{P}(w;Q_{1}(0))<\inf_{w\in u_{0}+W^{1,\infty}_{0}(Q_{1}(0))}\mathcal{P}(w;Q_{1}(0)),$$ for some $u_{0}\in W^{1,\infty}(Q_{1}(0))$, as soon as 
\eqn{pq1}
$$\textnormal{q}-\textnormal{p}>1.$$ Observe that the integrand in \eqref{i.7} is uniformly elliptic, as the related ellipticity ratio \eqref{i.2} is uniformly bounded. Building on Zhikov's breakthrough, Esposito \& Leonetti \& Mingione \cite{elm04}, Fonseca \& Malý \& Mingione \cite{fmm04}, and Balci \& Diening \& Surnachev \cite{bds20,bds25} proved\footnote{The results in \cite{elm04,fmm04,bds20,bds25} hold in the more general case of H\"older continuous coefficients, see also \cite{ddp26,fps26} for the nonautonomous, $\mu$-elliptic case.} that the solution $u\in u_{0}+W^{1,\textnormal{p}}_{0}(Q_{1}(0))$ to Dirichlet problem
$$
u_{0}+W^{1,\textnormal{p}}_{0}(Q_{1}(0))\ni u\mapsto \min_{w\in u_{0}+W^{1,\textnormal{p}}_{0}(Q_{1}(0))}\mathcal{P}(w;Q_{1}(0)),
$$
does not belong to $W^{1,\textnormal{q}}_{\loc}(Q_{1}(0))$. In particular, thanks to the geometry of coefficient $a$ in \eqref{i.7}, the set of essential discontinuity points of solutions in \cite{fmm04,bds25,ddp26} is a Cantor-type fractal of maximal Hausdorff dimension, which in turn implies that minima do not belong to $W^{1,t}_{\loc}(Q_{1}(0))$ for all $t>\textnormal{p}$. Although built on different models, the examples in \cite{gms79b,bil03,fps26} and \cite{zhi95,zhi97,elm04,fmm04,bds20,bds25,ddp26} share a common aspect: through their set of minimum points, coefficients identify interior sets where singular variation or large gradients are energetically less expensive. In other terms, coefficients prescribe the singular set of minima. This is one of the reasons why constructing irregular, scalar minimizers of autonomous functionals (resp. singular energy solution to autonomous elliptic PDEs) is so delicate. The autonomous case is instead a different story. There is no coefficient depending on (x) whose geometry can suggest, localize, or prescribe beforehand a prospective singular set: the latter has to emerge entirely from the equation itself, and it is therefore not even clear a priori where to look for it. To the best of our knowledge, the only scalar precedents relevant here are the unbounded constructions initiated by Giaquinta \cite{gia87} and Marcellini \cite{mar89,mar91}, and subsequently refined by Hong \cite{hon92}. The ultimate outcome, in dimension $n\ge 6$, is a minimum of functional 
$$
W^{1,2}(Q_{1}(0))\ni w\mapsto \int_{Q_{1}(0)}\snr{Dw}^{2}+\snr{\partial_{n}w}^{4}\dx,
$$
which is unbounded on the line $(0,\cdots,0,x_{n})\subset Q_{1}(0)$. At the zero-order level, Marcellini's counterexamples \cite{mar91} and the endpoint result of Hirsch \& Schäffner \cite{hs21} determined the optimal condition for the local boundedness of minima of superlinear, $(\textnormal{p},\textnormal{q})$-nonunformly elliptic, functionals \eqref{i.3}, that is
$$
\textnormal{q}\le
\begin{cases}
\displaystyle
\  \frac{(n-1)\textnormal{p}}{n-1-\textnormal{p}}\quad &\mbox{if} \ \ \textnormal{p}<n-1 \vspace{1.5mm}\\
\displaystyle
\ \infty\quad &\mbox{if} \ \ \textnormal{p}\ge n-1,
\end{cases}
$$
so that, via the common power-type parametrization \eqref{i.3}, \cite{mar91,hs21} and Theorem \ref{t.2} yields the optimal bound
$$
\textnormal{q}-\textnormal{p}<2,\qquad\mbox{and}\qquad  \textnormal{q}-\textnormal{p}\le \frac{\textnormal{p}^{2}}{n-1-\textnormal{p}} \quad \mbox{if} \ \ \textnormal{p}<n-1
$$
in superlinear, power-type nonuniformity. Needless to say, the examples in \cite{gia87,mar89,mar91,hon92} are intrinsically zero-order as singularities extend up to the boundary, otherwise maximum principle applies ensuring boundedness.

\subsection{Techniques}\label{tech} The construction of scalar non-Lipschitz minima of autonomous, strictly convex integrals presents a number of basic obstructions. The first, and actually most substantial one, comes from autonomy. Indeed, unlike in nonautonomous examples, there is no coefficient that can suggest, or encode a priori, the position and geometry of the singular set: such a set has to emerge from the equation itself, and it is therefore not even clear where to look for it. On top of this, global methods provide gradient bounds whenever suitable barriers are available. Minimality also comes along with full energy and comparison information, leaving none of the flexibility available to arbitrary weak solutions, and general partial regularity results already force the singular set to be small. Altogether, these facts leave very few entry points for a singular construction. Nevertheless, Theorem \ref{t.2} succeeds in delivering non-(locally) Lipschitz minima of integral $\mathcal{A}_{p}$, $2<p<\infty$, precisely above the nonuniformity threshold predicted by Ladyzhenskaya \& Ural'tseva. The key idea behind the proof is the duality (up to a sign and on suitable function classes) between integral $\mathcal{A}_{p}$ in \eqref{0.1} and the singular/degenerate Born-Infeld \cite{bs82,bar87} type energy \eqref{0.0}. Specifically, for $2<p<\infty$, $\mathcal{A}_{p}$ relates by duality to
\eqn{i.9}
$$
\begin{array}{c}
\displaystyle
\mathbb{X}(\Omega)\ni w\mapsto \int_{\Omega} (1-\snr{Dw}^{q})^{\frac{1}{q}}\dx\vspace{1.5mm}\\
\displaystyle
q:=\frac{p}{p-1} \ \stackrel{p>2}{\Longrightarrow} \  1<q<2,
\end{array}
$$
whose Euler-Lagrange equation (formally) is
\eqn{i.10}
$$
-\diver\left(\frac{\snr{Du}^{q-2}Du}{(1-\snr{Du}^{q})^{1-1/q}}\right)=0\qquad \mbox{in} \ \ \Omega,
$$
see \cite[Chapter V]{et99}, where convex duality is used in the analysis of existence and regularity properties of nonparametric minimal surfaces. Despite being autonomous, the very structure of \eqref{i.9} explicitly displays the singular set of maximizers, the so-called set of light rays, that is the family of segments along which maxima are affine, see Section \ref{dbi.s} for more details. Most importantly to our ends, the set of light rays does not refer to singularities of the maximum, that must be $1$-Lipschitz continuous to be admissible, but to those lines on which the related stress field blows up. In particular, being $q<2$ in \eqref{i.9}, the stress field may be integrable even if light rays exist. This is precisely what is relevant to us, in the convex duality perspective. In fact, in dimension $n=2$ - variables $(x,y)$ - we prescribe that our maximizer is a controlled, quadratic perturbation in the transversal $y$-direction of an affine map, and prove that such a perturbation exists via Cauchy-Kovalewskaya theorem. The resulting maximizer features a unique, maximally extended light ray contained in the line $\{y=0\}$, and the stress field has an integrable singularity behaving $\approx\snr{y}^{-2/p}$. We then manipulate the stress field from \eqref{i.10} to construct a non-Lipschitz, globally H\"older continuous minimum of $\mathcal{A}_{p}$ with interior singularities, whose gradient blows up precisely on the light ray. Cauchy-Kovalewskaya theorem plays a central role in our work as it gives the aforementioned perturbation function. It was indeed used in the construction of several examples of irregular solutions to certain classes of PDEs, see Nadirashvili \& Vlăduţ \cite{nv10} and Wang \& Yuan \cite{wy13} for the Special Lagrangian Equations, Caffarelli \& Yuan \cite{cy22} for the Monge-Ampère equation, and Mooney \& Savin \cite{ms24} for Cartesian currents. Let us also mention that the duality nonparametric minimal surfaces/classical Born-Infeld model, i.e., \eqref{i.9} with $q=2$, was used by Akamine \& Umehara \& Yamada \cite{auy20} to improve Calabi’s Bernstein-type theorem \cite{cal68} for zero Lorentizian mean curvature, entire graphs. Back to the proof of Theorem \ref{t.2}, the final, important part is the construction of the domain and of the boundary datum for the Dirichlet problem in \eqref{pd}. Since the minimum just constructed is merely Hölder continuous, restricting it to the boundary of a smooth domain crossing $\{y=0\}$ transversely would not produce a smooth datum. We therefore design a bounded, smooth, nonconvex domain whose boundary has infinite-order contact with $\{y=0\}$, using a signed exponentially flat profile as a part of the boundary, close to $\{y=0\}$, so to prevent the blow up of higher derivatives of the restriction of our minimum on that region of the boundary, cf. \cite{hir18, gai25}. The constructed domain still contains part of $\{y=0\}$, so the interior gradient singularities persist.
\subsubsection*{Organization of the paper.} This paper is organized as follows. Section \ref{pre} collects the preliminary material, including notation, the check of Ladyzhenskaya-Ural'tseva assumptions as $p$ variates in \eqref{0.1}, a brief discussion on the $BV$-relaxation of $\mathcal A_p$; and the Born-Infeld duality, together with basic existence results. Section \ref{lr.s} constructs the singular Born-Infeld maximizer with one light ray and proves that its stress satisfies the Euler equation across the ray. Section \ref{nl.m} integrates the rotated stress to obtain the two-dimensional non-Lipschitz minimizer and determines its sharp Sobolev-Marcinkiewicz regularity. Section \ref{pt.1} constructs the smooth domain and smooth boundary datum, lifts the example to every dimension $n\geq2$, and completes the proof of Theorem \ref{t.2}.
\section{Preliminaries}\label{pre}
\subsection{Notation} Throughout the paper, unless otherwise stated, $\Omega\subset\mathbb{R}^{n}$, $n\ge 2$, denotes a bounded Lipschitz domain. Any stronger regularity assumption on $\partial\Omega$ will be specified when needed. We use $\snr{\ \cdot\ }$ both for the Euclidean norm and, when applied to a measurable set, for its Lebesgue measure; $\mathcal{H}^{m}(\cdot)$, $0\le m\le n$, denotes the $m$-dimensional Hausdorff measure. The symbol $c\ge 1$ denotes a generic constant that may change at each
occurrence, whereas constants that need to be distinguished are denoted by
$c_{*}$, $\tilde c$, and so forth. Relevant parameter dependencies are
displayed in parentheses. For nonnegative quantities $a,b$, we write
$a\lesssim_{\gamma}b$ iff $a\le c(\gamma)b$. Moreover, $a\gtrsim_{\gamma}b$ means $b\lesssim_{\gamma}a$, while $a\approx_{\gamma}b$ means that both inequalities hold. For $x_{0}\in\mathbb{R}^{n}$ and $r>0$, we set
$B_{r}(x_{0}):=\{x\in\mathbb{R}^{n}\colon \snr{x-x_{0}}<r\}$ and $Q_{r}(x_{0}):=\{x\in \mathbb{R}^{n}\colon \max_{i\in \{1,\cdots,n\}}\snr{x_{i}-x_{0;i}}<r\}$. For $f\in L^{1}(B_{r}(x_{0}),\mathbb{R}^{k})$, $k\ge 1$, we
denote its integral average by
$$
(f)_{B_{r}(x_{0})}
:=\mint_{B_{r}(x_{0})}f(x)\dx
:=\snr{B_{r}(x_{0})}^{-1}
  \int_{B_{r}(x_{0})}f(x)\dx.
$$
We use the following uniform $\tx{O}$-notation. Given a set
$E\subset\mathbb{R}^{n}$ and functions $f_{1},f_{2}\colon E\to\mathbb{R}$,
we write $f_{2}=\tx{O}(f_{1})$ on $E$ if there exists a constant $c>0$,
independent of $x\in E$, such that $\snr{f_{2}(x)}\le c\snr{f_{1}(x)}$ for all $x\in E$. We further introduce the unit-gradient class
$$
\mathbb{X}(\Omega)
:=\left\{w\in W^{1,\infty}(\Omega)\colon
\nr{Dw}_{L^{\infty}(\Omega)}\le 1\right\}.
$$
Its elements are identified with their continuous representatives on
$\bar{\Omega}$. Finally, for $1<m<\infty$, $0<t\le \infty$, and a measurable map
$f\colon E\to\mathbb{R}^{k}$, $k\ge 1$, let
$$
[f]_{m,t;E}
:=\begin{cases}
\displaystyle
\ \left(m\int_{0}^{\infty}\left(\lambda^{m}\snr{\{x\in E\colon\snr{f(x)}>\lambda\}}\right)^{\frac{t}{m}}\frac{\d\lambda}{\lambda}\right)^{1/t}\quad &\mbox{if} \ \ t\in (0,\infty)\vspace{1.5mm}\\
\displaystyle
\ \sup_{\lambda>0}
\lambda \snr{\{x\in E\colon \snr{f(x)}>\lambda\}}^{1/m}\quad &\mbox{if} \ \ t=\infty.\end{cases}
$$
We then define the Sobolev-Lorentz space by
$$
W^{1;m,t}(\Omega)
:=\left\{w\in W^{1,1}(\Omega)\colon
[Dw]_{m,t;\Omega}<\infty\right\}.
$$
\subsection{Degenerate/singular nonparametric area integrals}\label{relax} In this section we deepen on Ladyzhenskaya \& Ural'tseva's \cite{lu70} sufficient conditions  for interior pointwise gradient regularity in nonuniformly elliptic PDEs. We show that the area type integral in \eqref{0.1} is covered whenever $1<p\le 2$, while if $p>2$ the crucial balance between ellipticity and nonuniformity is violated. This indeed leads to the construction of a non-Lipschitz solution (with interior singularities) to Dirichlet problem \eqref{pd}. For clarity, we refer to integrand \eqref{0.97} as singular if $1<p<2$, or degenerate when $2<p<\infty$. We conclude this part with a brief discussion on the BV-extension of functional \eqref{0.1} in relation to the existence of solutions to problem \eqref{pd}.
 
\subsubsection{Ladyzhenskaya \& Ural'tseva's conditions}\label{luc.s} We show that, for $1<p\le 2$, functional $\mathcal{A}_{p}$ matches the sufficient conditions listed in \cite[Theorems 1 or 4]{lu70}, while if $p>2$, all assumptions of \cite[Theorem 1]{lu70} are satisfied but one, and local Lipschitz regularity of solutions fails. We start by observing that if $p=2$, $\mathcal{A}_{p}$ reduces to the classical nonparametric area integral, that is the basic model on which \cite[Part II]{lu70} builds on, therefore all the assumptions of \cite[Theorem 4]{lu70} are automatically satisfied. We then look at the applicability of \cite[Theorem 1]{lu70}, in the sense of a priori estimates for bounded solutions to equations of the type
\eqn{0.91}
$$
-\diver(\partial A_{p}(Dv))\stackrel{\eqref{0.97}}{=}-\diver\left(\frac{\snr{Dv}^{p-2}Dv}{(1+\snr{Dv}^{p})^{1-1/p}}\right)=0, \qquad \quad 1<p<2,
$$
that is the Euler-Lagrange equation of functional $\mathcal{A}_{p}$. The forthcoming calculations are made for $\snr{z}>0$ as if $1<p<2$ in \eqref{0.97}, $A_{p}\in C^{\infty}_{\loc}(\mathbb{R}^{n}\setminus \{0\})\cap C^{1}_{\loc}(\mathbb{R}^{n})$, but this is not an issue since all hypotheses concerning the gradient variable will be imposed in the large, outside the origin. We have
$$
\begin{cases}
\displaystyle
\ \partial A_{p}(z)=(1+\snr{z}^{p})^{\frac{1}{p}-1}\snr{z}^{p-2}z\vspace{1.5mm}\\
\displaystyle
\ \partial^{2}A_{p}(z)=\frac{\snr{z}^{p-2}}{(1+\snr{z}^{p})^{1-1/p}}\left(\mathds{I}_{n}+(p-2)\left(\frac{z\otimes z}{\snr{z}^{2}}\right)-\frac{(p-1)\snr{z}^{p-2}(z\otimes z)}{1+\snr{z}^{p}}\right),
\end{cases}
$$
so that we can rewrite
$$
\partial^{2} A_{p}(z)=\mathcal{a}_{2}(\snr{z})\left(\mathds{I}_{n}-\frac{z\otimes z}{\snr{z}^{2}}\right)+\mathcal{a}_{1}(\snr{z})\left(\frac{z\otimes z}{\snr{z}^{2}}\right),
$$
where we set
$$\mathcal{a}_{2}(\snr{z}):= \frac{\snr{z}^{p-2}}{(1+\snr{z}^{p})^{1-1/p}}\qquad\mbox{and} \qquad 
\mathcal{a}_{1}(\snr{z}):=\frac{(p-1)\snr{z}^{p-2}}{(1+\snr{z}^{p})^{2-1/p}},
$$
whenever $\snr{z}\not =0$. Define also functions $$\mathcal{r}(\snr{z}):=\frac{\mathcal{a}_{1}(\snr{z})}{\mathcal{a}_{2}(\snr{z})}=\frac{p-1}{1+\snr{z}^{p}},\qquad \quad \mathcal{m}(\snr{z}):=\frac{\snr{z}^{p}}{1+\snr{z}^{p}},$$
and introduce the equivalent renormalized matrix
$$
\tx{A}_{p}(z)\equiv \left(\tx{A}_{p}^{i,j}(z)\right)_{i,j\in \{1,\cdots,n\}}:=\frac{\partial^{2}A_{p}(z)}{\mathcal{a}_{2}(\snr{z})}=\mathds{I}_{n}-\left(1-\mathcal{r}(\snr{z})\right)\left(\frac{z\otimes z}{\snr{z}^{2}}\right).
$$
We prove that \cite[Theorem 1]{lu70} holds true with the choices
\eqn{const}
$$
\begin{array}{c}
\displaystyle
\nu=1,\quad \quad c_{2}=n,\quad \quad M_{1}:=\left(\frac{1}{p-1}+\frac{p^{2}}{2}\right)^{\frac{1}{p}},\quad \quad \mu_{1}=1,\quad \quad c_{8}=c_{9}=c_{10}=c_{11}=0\vspace{1.5mm}\\
\displaystyle
 \nu_{0}:=(p-1)\mathcal{m}(M_{1}),\quad \quad k=2-p,\quad \quad c_{6}:=p\mathcal{m}(M_{1}),\quad \quad c_{7}:=p\sqrt{\mathcal{r}(M_{1})}\mathcal{m}(M_{1}).
\end{array}
$$
Specifically, we verify the validity of \cite[Theorem 1, (1.10), ($\widetilde{1.7}$), ($\widetilde{1.9}$), (1.12), (1.17), (1.18), (1.19), (1.34)]{lu70} - observe that \cite[Theorem 1, ($\widetilde{1.8}$), (1.20), (1.21), (1.29), (1.30)]{lu70} trivially hold as they refer to ingredients (coefficients, forcing terms, etc.) while \eqref{0.91} is autonomous. We bound
$$
\left(\sum_{i,j=1}^{n}(\tx{A}_{p}^{i,j}(z))^{2}\right)^{1/2}\le \left(n-1+\mathcal{r}(\snr{z})^{2}\right)^{1/2}\le n,
$$
so \cite[(1.10)]{lu70} holds. Differentiating further $\tx{A}_{p}$ we obtain
\eqn{0.81}
$$
\partial_{z_{k}}\tx{A}_{p}^{i,j}(z)=\mathcal{r}'(\snr{z})\left(\frac{z_{i}z_{j}z_{k}}{\snr{z}^{3}}\right)+\left(\frac{\mathcal{r}(\snr{z})-1}{\snr{z}}\right)\left(\left(\delta_{i,k}-\frac{z_{i}z_{k}}{\snr{z}^{2}}\right)\frac{z_{j}}{\snr{z}}+\left(\delta_{j,k}-\frac{z_{j}z_{k}}{\snr{z}^{2}}\right)\frac{z_{i}}{\snr{z}}\right),
$$
and
\begin{eqnarray*}
\left(\sum_{i,j,k=1}^{n}\left(\partial_{z_{k}}\tx{A}_{p}^{i,j}(z)\right)^{2}\right)^{1/2}&=&\left(\mathcal{r}'(\snr{z})^{2}+\frac{2(n-1)(\mathcal{r}(\snr{z})-1)^{2}}{\snr{z}^{2}}\right)^{1/2}\nonumber \\
&=&\left(\left(\frac{p\mathcal{m}(\snr{z})\mathcal{r}(\snr{z})}{\snr{z}}\right)^{2}+\frac{2(n-1)(\mathcal{r}(\snr{z})-1)^{2}}{\snr{z}^{2}}\right)^{1/2}\nonumber \\
&\le&\frac{\sqrt{p^{2}+2(n-1)}}{\snr{z}}\le \frac{2\sqrt{n}}{\snr{z}}\le \frac{4\sqrt{c_{2}}}{\snr{z}},
\end{eqnarray*}
and \cite[($\widetilde{1.7}$)]{lu70} is verified. Next, we introduce
\eqn{0.82}
$$\mathcal{E}_{\tx{A}_{p}}(z):=\sum_{i,j=1}^{n}\tx{A}_{p}^{i,j}(z)z_{i}z_{j}=\mathcal{r}(\snr{z})\snr{z}^{2}=(p-1)\mathcal{m}(\snr{z})\snr{z}^{2-p},$$ 
and compute
\eqn{0.80}
$$
\partial_{z}\mathcal{E}_{\tx{A}_{p}}(z)=\frac{(p-1)z}{1+\snr{z}^{p}}\left(2-\frac{p\snr{z}^{p}}{1+\snr{z}^{p}}\right)=z\mathcal{r}(\snr{z})\left(2-p\mathcal{m}(\snr{z})\right).
$$
We then estimate
\begin{eqnarray*}
\left(\sum_{i=1}^{n}\left(\partial_{z_{i}}\mathcal{E}_{\tx{A}_{p}}(z)\right)^{2}\right)^{\frac{1}{2}}&=&\left(\snr{z}^{2}\mathcal{r}(\snr{z})^{2}\left(2-p\mathcal{m}(\snr{z})\right)^{2}\right)^{1/2}\nonumber \\
&=&\sqrt{\mathcal{r}(\snr{z})\mathcal{E}_{\tx{A}_{p}}(z)}\left(2-p\mathcal{m}(\snr{z})\right)\le 2\sqrt{\mathcal{E}_{\tx{A}_{p}}(z)} \le 4\sqrt{c_{2}}\sqrt{\mathcal{E}_{\tx{A}_{p}}(z)},
\end{eqnarray*}
which is \cite[$\widetilde{(1.9)}$]{lu70}. Furthermore, by \eqref{0.80} we have
$$
\frac{\langle z,\partial_{z}\mathcal{E}_{\tx{A}_{p}}(z)\rangle}{\mathcal{E}_{\tx{A}_{p}}(z)}=2-p\mathcal{m}(\snr{z}).
$$
Now, with $M_{1}$ as in \eqref{const} observe that
\eqn{0.93}
$$
\snr{z}\ge M_{1} \ \Longrightarrow \ \begin{cases}
    \displaystyle
    \ \mathcal{E}_{\tx{A}_{p}}(z)\ge (p-1)\mathcal{m}(M_{1})\snr{z}^{2-p}=\nu_{0}\snr{z}^{k}\vspace{1.5mm}\\
    \displaystyle
    \ 2-\frac{\langle z,\partial_{z}\mathcal{E}_{\tx{A}_{p}}(z)\rangle}{\mathcal{E}_{\tx{A}_{p}}(z)}=p\mathcal{m}(\snr{z})\ge p\mathcal{m}(M_{1})=c_{6},
\end{cases}
$$
and \cite[(1.12) and (1.17)]{lu70} hold true. Keep $\eqref{0.93}_{1}$ in mind, as it encodes compensation between ellipticity and nonuniformity and as such it is the key source of (ir)regularity. Moreover, 
\begin{eqnarray*}
\langle\tx{A}_{p}(z)\xi,\xi\rangle&=&\snr{\xi}^{2}\left(1-(1-\mathcal{r}(\snr{z}))\left|\left\langle\frac{z}{\snr{z}},\frac{\xi}{\snr{\xi}}\right\rangle\right|^{2}\right)\nonumber \\
&\ge&\snr{\xi}^{2}\left(1-(1-\mathcal{r}(\snr{z}))\left|\left\langle\frac{z}{\snr{z}},\frac{\xi}{\snr{\xi}}\right\rangle\right|\right)\ge \snr{\xi}^{2}\left(1-\mu_{1}\left|\left\langle\frac{z}{\snr{z}},\frac{\xi}{\snr{\xi}}\right\rangle\right|\right),
\end{eqnarray*}
and \cite[(1.18)]{lu70} follows. We look back at \eqref{0.81} and estimate
\begin{eqnarray*}
\left(\sum_{i,j=1}^{n}\left(\sum_{k=1}^{n} z_{k}\partial_{z_{k}}\tx{A}_{p}^{i,j}(z)\right)^{2}\right)^{1/2}&=&\left(\frac{\mathcal{r}'(\snr{z})^{2}}{\snr{z}^{6}}\left(\sum_{k=1}^{n}z_{k}^{2}\right)^{2}\left(\sum_{i,j=1}^{n}(z_{i}z_{j})^{2}\right)\right)^{1/2}\nonumber \\
&=&\snr{\mathcal{r}'(\snr{z})}\snr{z}=p\mathcal{r}(\snr{z})\mathcal{m}(\snr{z})\nonumber \\
&\stackrel{\eqref{0.82}}{=}&p\left(\sqrt{\mathcal{r}(\snr{z})}\mathcal{m}(\snr{z})\right)\left(\frac{\sqrt{\mathcal{E}_{\tx{A}_{p}}(z)}}{\snr{z}}\right)\le \left(p\sqrt{\mathcal{r}(M_{1})}\mathcal{m}(M_{1})\right)\left(\frac{\sqrt{\mathcal{E}_{\tx{A}_{p}}(z)}}{\snr{z}}\right),
\end{eqnarray*}
which yields \cite[(1.19)]{lu70} - here we used also that $t\mapsto \sqrt{\mathcal{r}(t)}\mathcal{m}(t)$ is decreasing for $t>2^{1/p}$ and that\footnote{If $1<p<3/2$ it is $M_{1}^{p}>(p-1)^{-1}>2$, while if $2>p\ge 3/2$ we have $M_{1}^{p}\ge (p-1)^{-1}+9/8>1+9/8>2$.} $M_{1}\ge 2^{1/p}$. Finally, notice that
\begin{eqnarray*}
c'&:=&c_{6}-\frac{c_{7}^{2}}{2}=1+\frac{(p-1)M_{1}^{p}-1}{1+M_{1}^{p}}-\frac{1}{2}\left(\frac{p^{2}(p-1)M_{1}^{2p}}{(1+M_{1}^{p})^{3}}\right)\nonumber \\
&=&1+\frac{p^{2}(p-1)}{2(1+M_{1}^{p})}\left(1-\frac{M_{1}^{2p}}{(1+M_{1}^{p})^{2}}\right)\ge 1,
\end{eqnarray*}
and \cite[(1.34)]{lu70} is satisfied. Consequently, \cite[Theorem 1]{lu70} guarantees a priori interior gradient estimates for solutions to \eqref{0.91} for $1<p<2$. Next, if $p>2$ a direct calculation highlights that all assumptions of \cite[Theorem 1]{lu70} check out, possibly with larger constants (in particular $M_{1}^{p}>2$ and $\snr{z}\ge M_{1} \ \Longrightarrow \ \mathcal{r}(\snr{z})\le 1$ for all $1<p<\infty$), regardless the size of $p$, except one, that is \cite[(1.12)]{lu70} - our $\eqref{0.93}_{1}$ - as now $k=2-p<0$. Recall that for Lipschitz regularity the only relevant ellipticity information is the one available at infinity, so we can assume $\snr{z}\ge M_{1}$ and record that in this case it is $\mathcal{r}(\snr{z})\approx_{p}\snr{z}^{-p}$. By \eqref{0.82} we then have
\eqn{0.92}
$$
\mathcal{E}_{\tx{A}_{p}}(z)=\mathcal{r}(\snr{z})\snr{z}^{2}\approx_{p}\left(\mathcal{r}(\snr{z})^{-1}\right)^{\frac{2}{p}-1}\stackrel{p>2}{\to}_{\snr{z}\to \infty}0,
$$
so that ellipticity decreases when the nonuniformity rate of \eqref{0.91} increases, a clear obstruction to regularity.

\subsubsection{BV-Relaxation} Functional $\mathcal{A}_{p}$ in \eqref{0.1} features linear growth, that is
\eqn{0.96}
$$
\snr{z}\le A_{p}(z)\le 1+\snr{z}\qquad \mbox{for all} \ \ z\in \mathbb{R}^{n},
$$
so minimizing sequences\footnote{Subject to \eqref{0.94}, a minimizing sequence for $\mathcal{A}_{p}$ in $\mathcal{D}_{0}(v_{0};\Omega)$ is a sequence $\{v_{i}\}_{i\in \mathbb{N}}\subset \mathcal{D}_{0}(v_{0};\Omega)$ such that $\mathcal{A}_{p}(v_{i};\Omega)\to \inf_{w\in \mathcal{D}_{0}(v_{0};\Omega)}\mathcal{A}_{p}(w;\Omega)$, cf. \cite[Definition 2.2]{bs15}.} in Dirichlet class $\mathcal{D}_{0}(v_{0};\Omega):=v_{0}+W^{1,1}_{0}(\Omega)$ for some $v_{0}\in W^{1,1}(\Omega)$, are bounded in $W^{1,1}(\Omega)$, but not necessarily weakly precompact there. This means that minima might not exist in $\mathcal{D}_{0}(v_{0};\Omega)$, and it is therefore necessary to extend $\mathcal{A}_{p}$ by semicontinuity in a larger space, $BV(\Omega)$, see \cite{gms79a,bs13,bs15,gk24}. Since 
\eqn{0.94}
$$
A_{p}\mbox{ is strictly convex and } \lim_{\snr{z}\to \infty}\frac{A_{p}(z)}{\snr{z}}=1,
$$
following \cite[Section 2]{bs15}, the BV-extension by semicontinuity of $\mathcal{A}_{p}$ in class $\mathcal{D}_{0}(v_{0};\Omega)$, that is, for any\footnote{When working with $BV$-maps $w$, it is customary to adopt the decomposition of the gradient measure $Dw=\nabla w \mathcal{L}^{n}+D^{s}w$ into its absolutely continuous and its singular part with respect to the $n$-dimensional Lebesgue measure $\mathcal{L}^{n}$. Needless to say, for $W^{1,1}$-functions $w$ it is $Dw=\nabla w \mathcal{L}^{n}$ and $D^{s}w=0$, and we shall always refer to the gradient of a $W^{1,1}$-function with the symbol $Dw$.} $w\in BV(\Omega)$,
$$
\bar{\mathcal{A}}_{p}(w;\mathcal{D}_{0}(v_{0};\Omega)):=\inf\left\{\liminf_{i\to \infty}\mathcal{A}_{p}(w_{i};\Omega)\colon \{w_{i}\}_{i\in \mathbb{N}}\subset \mathcal{D}_{0}(v_{0};\Omega)\mbox{ and }w_{i}\to w\mbox{ in }L^{1}(\Omega)\right\}
$$
admits integral representation
\eqn{0.95}
$$
\bar{\mathcal{A}}_{p}(w;\mathcal{D}_{0}(v_{0};\Omega))=\int_{\Omega}A_{p}(\nabla w)\dx+\snr{D^{s}w}(\Omega)+\int_{\partial\Omega}\snr{w-v_{0}}\d\mathcal{H}^{n-1}(x),
$$
where functions in the last integral of \eqref{0.95} are intended as traces, and we also used that the recession function associated to $A_{p}$ is
$$
A_{p}^{\infty}(z):=\lim_{s\to \infty}\frac{A_{p}(sz)}{s}=\snr{z}\qquad \mbox{for all} \ \ z\in \mathbb{R}^{n}.
$$
Next, we recall that a generalized minimizer of $\mathcal{A}_{p}$ in class $\mathcal{D}_{0}(v_{0};\Omega)$ is a function $\bar{v}\in BV(\Omega)$ such that $\bar{\mathcal{A}}_{p}(\bar{v};\mathcal{D}_{0}(v_{0};\Omega))\le \bar{\mathcal{A}}_{p}(w;\mathcal{D}_{0}(v_{0};\Omega))$ for all $w\in BV(\Omega)$, cf. \cite[Definition 2.1]{bs15}. The existence of generalized minimizers is by-now classical, cf. \cite[Section 2]{gms79a}.
\begin{proposition}
Assume \eqref{0.96}-\eqref{0.94}, and let $v_{0}\in W^{1,1}(\Omega)$ be a function. Then there exists a generalized minimizer of $\mathcal{A}_{p}$ in class $\mathcal{D}_{0}(v_{0};\Omega)$.
\end{proposition}
\noindent We conclude this section with an auxiliary result, probably well-known.
\begin{proposition}\label{rel.p}
Let $v_{0}\in W^{1,1}(\Omega)$ be a function and assume that there exists a solution $v\in \mathcal{D}_{0}(v_{0};\Omega)$ to Dirichlet problem \eqref{pd}. Then $v$ is the unique solution to \eqref{pd}, the unique generalized minimizer of $\mathcal{A}_{p}$ in class $\mathcal{D}_{0}(v_{0};\Omega)$, and $\bar{\mathcal{A}}_{p}(v;\mathcal{D}_{0}(v_{0};\Omega))=\mathcal{A}_{p}(v;\Omega)$.
\end{proposition}
\begin{proof}
First, notice that $\eqref{0.94}_{1}$ guarantees that $v\in \mathcal{D}_{0}(v_{0};\Omega)$ is the unique solution to \eqref{pd} in $\mathcal{D}_{0}(v_{0};\Omega)$. Moreover, via \eqref{0.95} we have that $\mathcal{A}_{p}(w;\Omega)=\bar{\mathcal{A}}_{p}(w;\mathcal{D}_{0}(v_{0};\Omega))$ for all $w\in \mathcal{D}_{0}(v_{0};\Omega)$, thus $\mathcal{A}_{p}(v;\Omega)=\bar{\mathcal{A}}_{p}(v;\mathcal{D}_{0}(v_{0};\Omega))$. By \eqref{0.94}, \cite[Theorem 2.3]{bs15} applies with constant minimizing sequence $\{v_{i}\}_{i\in \mathbb{N}}\subset \mathcal{D}_{0}(v_{0};\Omega)$, $v_{i}=v$ for all $i\in \mathbb{N}$, so that 
$$
\mathcal{A}_{p}(v;\Omega)=\inf_{w\in \mathcal{D}_{0}(v_{0};\Omega)}\mathcal{A}_{p}(w;\Omega)=\inf_{w\in BV(\Omega)}\bar{\mathcal{A}}_{p}(w;\mathcal{D}_{0}(v_{0};\Omega))=\bar{\mathcal{A}}_{p}(v;\mathcal{D}_{0}(v_{0};\Omega)),
$$
and \cite[Proposition B.1]{bs13} allows concluding that $v$ is also the unique generalized minimizer of $\mathcal{A}_{p}$.
\end{proof}
\subsection{Degenerate/singular Born-Infeld functionals}\label{dbi.s} Here we give a quick overview of the main properties of Born-Infeld functionals with emphasis on the existence of maxima, their "singular" set, and to the application of convex duality arguments. For $1<q<\infty$, introduce functional
\eqn{0.0}
$$
\mathbb{X}(\Omega)\ni w\mapsto\mathcal{J}_{q}(w;\Omega):=\int_{\Omega}(1-\snr{Dw}^{q})^{\frac{1}{q}}\dx,
$$
governed by integrand $\bar{B}_{1}(0)\ni z\mapsto J_{q}(z):=(1-\snr{z}^{q})^{\frac{1}{q}}$. For definiteness, we agree that $J_{q}$ is degenerate if $2<q<\infty$ or singular when $1<q<2$. Existence and uniqueness for solutions to the Dirichlet problem for functional $\mathcal{J}_{q}$ are standard after Bartnik \& Simon \cite{bs82}.
\begin{lemma}
Let $\Omega\subset \mathbb{R}^{n}$ be a bounded, Lipschitz domain and $u_{0}\in C(\partial \Omega)$ be a function. Introduce class
\eqn{0.33}
$$
\mathcal{D}_{\infty}(u_{0};\Omega):=\left\{w\in \mathbb{X}(\Omega)\colon w=u_{0}\mbox{ on }\partial \Omega\right\}.
$$
Then there exists a unique maximizer $u\in \mathcal{D}_{\infty}(u_{0};\Omega)$ of Dirichlet problem
$$
\mathcal{D}_{\infty}(u_{0};\Omega)\ni w\mapsto \max_{w\in \mathcal{D}_{\infty}(u_{0};\Omega)}\mathcal{J}_{q}(w;\Omega).
$$
iff $\mathcal{D}_{\infty}(u_{0};\Omega)\not =\emptyset$. 
\end{lemma}
\begin{proof}
Since $J_{q}$ is continuous and strictly concave on
$\bar{B}_{1}(0)$, and
$\mathcal D_\infty(u_0;\Omega)$ is uniformly bounded and
equicontinuous - recall that, being $\Omega$ Lipschitz-regular, the elements of $\mathbb{X}(\Omega)$ are continuous in $\bar{\Omega}$ - the compactness and upper-semicontinuity argument
of \cite[Proposition 1.1]{bs82} applies verbatim and yields existence. Uniqueness follows from the strict concavity of $J_{q}$. 
\end{proof}
\noindent Observe that $J_{q}\in C^{\infty}(B_{1}(0)\setminus \{0\})\cap C^{0,\frac{1}{q}}(\bar{B}_{1}(0))$, so the natural "singular set" for maxima of $\mathcal{J}_{q}$ is the so-called set of light rays \cite{bs82,bar87}, that is, the union of those maximally extended lines verifying 
\eqn{def.light}
$$
\overline{xy}\subset \bar{\Omega}\mbox{ such that }x\not =y\in \bar{\Omega},\mbox{ }(\overline{xy}\setminus \{x,y\})\subset \Omega,\mbox{ and }\snr{u(x)-u(y)}=\snr{x-y}
$$
and
\eqn{def.light.1}
$$
\begin{array}{c}
\displaystyle
\mbox{for every }z \in\bar{\Omega}\setminus \overline{xy}\mbox{ collinear with }x,y\mbox{ such that the smallest segment containing }\overline{xy}\cup\{z\}\vspace{1.5mm}\\
\displaystyle
\mbox{ has relative interior contained in }\Omega,\mbox{ one has } \snr{u(x)-u(z)}<\snr{x-z}\mbox{ or }\snr{u(y)-u(z)}<\snr{y-z}.
\end{array}
$$
In other terms, we can define the set of light rays as $\Sigma:=\left\{\overline{xy}\subset \bar{\Omega}\colon \mbox{\eqref{def.light}-\eqref{def.light.1} hold}\right\}.$ We adopt the convention that if $x\not =y$ are the endpoints of a maximally extended light segment $\overline{xy}$, we label them so that $u(x)>u(y)$ and we call $x$ the left endpoint and $y$ the right endpoint. Notice that this position is valid for all functions $u\in \mathbb{X}(\Omega)$. It is in fact worth recalling that $u$ is differentiable at all interior points of the segments in $\Sigma$. Specifically, if $z\in \Omega\cap (\overline{xy}\setminus\{x,y\})$, by basic rigidity arguments it is $Du(z)=(x-y)/\snr{x-y}$, thus $\snr{Du(z)}=1$, see \cite[Lemma 3.5]{wan08}. Next, let us highlight that functional $\mathcal{J}_{q}$ is naturally related via convex duality \cite[Chapters IV-V]{et99} to integral $\mathcal{A}_{p}$. Set $\mathcal{j}_{q}(\snr{z}):=J_{q}(z)$, notice that $z\mapsto J_{q}(z)$ is concave, extend
\eqn{0.34}
$$
\ti{J}_{q}(z):=\begin{cases}
    \displaystyle
    \ -\jj_{q}(\snr{z})\quad &\mbox{if} \ \ \snr{z}\le 1\vspace{1.5mm}\\
    \displaystyle
    \ +\infty\quad &\mbox{if} \ \ \snr{z}>1,
\end{cases}
$$
and name
\eqn{0.35}
$$
W^{1,1}(\Omega)\ni w\mapsto \ti{\mathcal{J}}_{q}(w;\Omega):=\int_{\Omega}\ti{J}_{q}(Dw)\dx.
$$
As a consequence of the fact that $\mathcal{J}_{q}(w;\Omega)=-\ti{\mathcal{J}}_{q}(w;\Omega)$ for all $w\in \mathcal{D}_{\infty}(u_{0};\Omega)$ we have
\eqn{0.2}
$$
\max_{w\in \mathcal{D}_{\infty}(u_{0};\Omega)}\mathcal{J}_{q}(w;\Omega)=-\min_{w\in \mathcal{D}_{\infty}(u_{0};\Omega)}\ti{\mathcal{J}}_{q}(w;\Omega).
$$
A direct computation and the regularity of $J_{q}$ on $B_{1}(0)\setminus \{0\}$ show that
$$
\frac{\min\{q-1,1\}\snr{z}^{q-2}\snr{\xi}^{2}}{(1-\snr{z}^{q})^{1-1/q}}\le \langle \partial^{2}\ti{J}_{q}(z)\xi,\xi\rangle\qquad \mbox{and}\qquad \snr{\partial^{2}\ti{J}_{q}(z)}\le \frac{\sqrt{n-1+(q-1)^{2}}\snr{z}^{q-2}}{(1-\snr{z}^{q})^{2-1/q}},
$$
for all $z\in B_{1}(0)\setminus\{0\}$, $\xi\in \mathbb{R}^{n}$. Moreover, $\snr{z}\mapsto -\jj_{q}(\snr{z})$ is convex and continuous on $\bar{B}_{1}(0)$, the lower semicontinuous extension by $+\infty$ outside $\bar{B}_{1}(0)$ is a proper, closed, convex function. For the sake of the reader, let us explicitly compute the Fenchel conjugate of $\ti{J}_{q}$.
\begin{lemma}
The Fenchel conjugate of $\ti{J}_{q}$ is $A_{p}(z):=\ti{J}_{q}^{*}(z)=(1+\snr{z}^{p})^{\frac{1}{p}}$ for all $z\in \mathbb{R}^{n}$, with $p:=q/(q-1)$. Moreover, for all $z\in \mathbb{R}^{n}\setminus \{0\}$, $\xi\in B_{1}(0)\setminus\{0\}$ it holds
\eqn{0.98}
$$
\partial A_{p}(z)=\frac{\snr{z}^{p-2}z}{(1+\snr{z}^{p})^{1-1/p}},\qquad \quad \partial \ti{J}_{q}(\xi)=\frac{\snr{\xi}^{q-2}\xi}{(1-\snr{\xi}^{q})^{1-1/q}},
$$
and, in particular, 
\eqn{0.99}
$$
\partial A_{p}(\partial\ti{J}_{q}(\xi))=\xi\qquad \mbox{for all} \ \ \xi\in B_{1}(0)\setminus \{0\}.
$$
\end{lemma}
\begin{proof}
By definition and H\"older inequality with conjugate exponents $(p,q)$, for $z\in \mathbb{R}^{n}$ it is 
\begin{eqnarray*}
    \ti{J}_{q}^{*}(z)&:=&\sup_{\snr{\xi}\le 1}\{\langle z,\xi\rangle-\ti{J}_{q}(\xi)\}=\sup_{\snr{\xi}\le 1}\{\snr{\xi}\snr{z}+\jj_{q}(\snr{\xi})\}\nonumber \\
    &=&\sup_{\snr{\xi}\le 1}\{\langle(\snr{z},1),(\snr{\xi},\jj_{q}(\snr{\xi}))\rangle\}\le (1+\snr{z}^{p})^{\frac{1}{p}},
\end{eqnarray*}
where we used that $(\snr{\xi}^{q}+\jj_{q}(\snr{\xi})^{q})^{\frac{1}{q}}=1$. Equality above is attained at $\snr{\xi}=\snr{z}^{p/q}(1+\snr{z}^{p})^{-1/q}$, therefore
$$\xi_{z}=\begin{cases}
\displaystyle
\ \frac{\snr{z}^{p-2}z}{(1+\snr{z}^{p})^{1-1/p}}\quad &\mbox{if} \ \ z\not =0\vspace{1.5mm}\\
\displaystyle
\ 0\quad &\mbox{if} \ \ z=0,\end{cases}$$
attains the maximum in the definition of $\ti{J}_{q}^{*}$. Identities \eqref{0.98}-\eqref{0.99} directly follow from the definitions of $A_{p}$ and $\ti{J}_{q}$. The proof is complete.
\end{proof}

\section{Light rays in the singular Born-Infeld integral}\label{lr.s}
\noindent In this section we construct a local maximizer of functional $\mathcal{J}_{q}$ which is affine on a segment. This serves as a main building block for our example. Unless otherwise mentioned, we assume that $1<q<2$ in \eqref{0.0}, $p=q/(q-1)$, and $n=2$. 
\subsection{The stress tensor}\label{ss.s} We prescribe that our candidate maximizer of $\mathcal{J}_{q}$ in \eqref{0.0} is globally "almost affine", namely a controlled analytic perturbation of an affine profile along a segment. We then calculate the modulus of its gradient, isolate the quantity measuring displacement from the light ray, and factor this deficit into the square of the transverse variable times an analytic factor. This factorization helps deriving an explicit expression for the singular stress field, that in turn will be used to determine the unknown function in the ansatz. We look for maxima of functional $\mathcal{J}_{q}$ in \eqref{0.0} of the form
\eqn{0.18}
$$u(x,y):=x+y^{2}\gamma(x,y),$$
for some analytic function $\gamma$ to be identified later on. We compute 
\eqn{0.16}
$$
\partial_{x}u(x,y)=1+y^{2}\partial_{x}\gamma(x,y)\qquad \mbox{and}\qquad \partial_{y}u(x,y)=2y\left(\gamma(x,y)+\frac{y\partial_{y}\gamma(x,y)}{2}\right).
$$
For simplicity define
\eqn{0.7}
$$
\Gamma_{1}(x,y):=1+y^{2}\partial_{x}\gamma(x,y)\qquad \mbox{and}\qquad \Gamma_{2}(x,y):=\gamma(x,y)+\frac{y\partial_{y}\gamma(x,y)}{2},
$$
so that
$$
Du(x,y)=\left(\Gamma_{1}(x,y),2y\Gamma_{2}(x,y)\right)
$$
and
\begin{eqnarray}\label{0.4}
\snr{Du(x,y)}^{q}&=&\left(\snr{Du(x,y)}^{2}\right)^{\frac{q}{2}}=\left(\Gamma_{1}(x,y)^{2}+4y^{2}\Gamma_{2}(x,y)^{2}\right)^{\frac{q}{2}}\nonumber \\
&=&\left(1+2y^{2}\partial_{x}\gamma(x,y)+y^{4}(\partial_{x}\gamma(x,y))^{2}+4y^{2}\Gamma_{2}(x,y)^{2}\right)^{\frac{q}{2}}=\left(1+y^{2}\Gamma_{0}(x,y)\right)^{\frac{q}{2}},
\end{eqnarray}
where we set 
\eqn{0.15}
$$\Gamma_{0}(x,y):=2\partial_{x}\gamma(x,y)+y^{2}(\partial_{x}\gamma(x,y))^{2}+4\Gamma_{2}(x,y)^{2}.$$ With $(t,\tx{d})\in \mathbb{R}\times \mathbb{R}$, $\snr{t^{2}\tx{d}}<1$, we define
\eqn{0.5}
$$
\ti{\mathcal{C}}(t,\tx{d}):=-\frac{q\tx{d}}{2}\int_{0}^{1}(1+\tau t^{2}\tx{d})^{\frac{q-2}{2}}\d\tau,
$$
which is analytic for $\snr{t^{2}\tx{d}}<1$ as the expansion in binomial series 
$$
\left(1+t^{2}\tx{d}\right)^{\frac{q}{2}}=\sum_{i=0}^{\infty}\binom{q/2}{i}t^{2i}\tx{d}^{i},
$$
holds, so if $t\not =0$ we have,
\begin{eqnarray*}
\ti{\mathcal{C}}(t,\tx{d})&=&\frac{1-(1+t^{2}\tx{d})^{\frac{q}{2}}}{t^{2}}=-\sum_{i=1}^{\infty}\binom{q/2}{i}t^{2(i-1)}\tx{d}^{i}\nonumber \\
&=&-\frac{q\tx{d}}{2}-\left(\frac{q(q-2)}{8}\right)t^{2}\tx{d}^{2}+\tx{g}(t,\tx{d}),
\end{eqnarray*}
where we set $\tx{g}(t,\tx{d}):=-\sum_{i=3}^{\infty}\binom{q/2}{i}t^{2(i-1)}\tx{d}^{i}\equiv \tx{O}(t^{4}\tx{d}^{3})$, while in $t=0$ it is $\ti{\mathcal{C}}(0,\tx{d})=-q\tx{d}/2$, so that $\ti{\mathcal{C}}$ can be expanded as
$$
\ti{\mathcal{C}}(t,\tx{d}):=\begin{cases}
    \displaystyle
    \ \frac{1-(1+t^{2}\tx{d})^{\frac{q}{2}}}{t^{2}}\quad &\mbox{if} \ \ t\not =0\vspace{1.5mm}\\
    \displaystyle
    \ -\frac{q\tx{d}}{2}\quad &\mbox{if} \ \ t=0.
\end{cases}
$$
In a neighborhood of $(t,\tx{d})=(0,-2)$, $\ti{\mathcal{C}}$ is analytic and $\ti{\mathcal{C}}(0,-2)=q>0$. We also introduce function 
\eqn{0.3}
$$\ti{\mathcal{S}}(t,\tx{d}):=\frac{(1+t^{2}\tx{d})^{(q-2)/2}}{\ti{\mathcal{C}}(t,\tx{d})^{1/p}},$$
which is analytic again in a neighborhood of $(0,-2)$ given that both terms in \eqref{0.3} are positive. We then introduce compositions
\eqn{0.20}
$$
\mathcal{C}(x,y):=\ti{\mathcal{C}}(y,\Gamma_{0}(x,y))\qquad \mbox{and}\qquad \mathcal{S}(x,y):=\ti{\mathcal{S}}(y,\Gamma_{0}(x,y)).
$$
Merging \eqref{0.4}-\eqref{0.5}, we obtain the split
\eqn{0.6}
$$
1-\snr{Du(x,y)}^{q}=y^{2}\mathcal{C}(x,y).
$$
Next, recalling that $J_{q}(z)=-\ti{J}_{q}(z)$ for all $z\in \bar{\mathbb{B}}_{n}$, via \eqref{0.7}, \eqref{0.4}, \eqref{0.3}, and \eqref{0.6}, if $y\not =0$ we compute the stress tensor 
\begin{eqnarray}\label{0.26}
\partial \ti{J}_{q}(Du(x,y))&=&\frac{\snr{Du(x,y)}^{q-2}Du(x,y)}{(1-\snr{Du}^{q})^{1-1/q}}\nonumber \\
&=&\frac{(1+y^{2}\Gamma_{0}(x,y))^{(q-2)/2}}{(y^{2}\mathcal{C}(x,y))^{1-1/q}}\left(\Gamma_{1}(x,y),2y\Gamma_{2}(x,y)\right)\nonumber \\
&=&\left(\frac{\mathcal{S}(x,y)}{\snr{y}^{2/p}}\right)\left(\Gamma_{1}(x,y),2y\Gamma_{2}(x,y)\right),
\end{eqnarray}
and its divergence
\begin{eqnarray}\label{0.29}
\diver(\partial \ti{J}_{q}(Du(x,y)))&=&\snr{y}^{-2/p}\partial_{x}(\mathcal{S}(x,y)\Gamma_{1}(x,y))+\partial_{y}\left(\frac{2y\mathcal{S}(x,y)\Gamma_{2}(x,y)}{\snr{y}^{2/p}}\right)\nonumber \\
&=&2\left(1-\frac{2}{p}\right)\left(\frac{\mathcal{S}(x,y)\Gamma_{2}(x,y)}{\snr{y}^{2/p}}\right)+2y\snr{y}^{-\frac{2}{p}}\partial_{y}(\mathcal{S}(x,y)\Gamma_{2}(x,y))\nonumber \\
&&+\snr{y}^{-2/p}\partial_{x}(\mathcal{S}(x,y)\Gamma_{1}(x,y))=:\frac{\mathcal{R}[\gamma](x,y)}{\snr{y}^{2/p}}.
\end{eqnarray}
Above, we set
\eqn{0.8}
$$
\mathcal{R}[\gamma](x,y):=\partial_{x}(\mathcal{S}(x,y)\Gamma_{1}(x,y))+2\left(1-\frac{2}{p}\right)\mathcal{S}(x,y)\Gamma_{2}(x,y)+2y\partial_{y}\left(\mathcal{S}(x,y)\Gamma_{2}(x,y)\right),
$$
for $y\not =0$. Let us explicitly compute the various terms on the right-hand side of \eqref{0.8}. For shortness, in the forthcoming computations we will suppress the dependency on variables $(x,y)$. We expand
\begin{eqnarray}\label{0.8.1}
\partial_{x}(\mathcal{S}\Gamma_{1})&=&\Gamma_{1}\partial_{x}\mathcal{S}+\mathcal{S}\partial_{x}\Gamma_{1}=y^{2}\mathcal{S}\partial^{2}_{x}\gamma+\Gamma_{1}\partial_{\tx{d}}\ti{\mathcal{S}}\partial_{x}\Gamma_{0}\nonumber \\
&=&y^{2}\mathcal{S}\partial^{2}_{x}\gamma+2\Gamma_{1}\partial_{\tx{d}}\ti{\mathcal{S}}\left(\partial_{x}^{2}\gamma\left(1+y^{2}\partial_{x}\gamma\right)+4\Gamma_{2}\left(\partial_{x}\gamma+\frac{y\partial^{2}_{xy}\gamma}{2}\right)\right)\nonumber \\
&=&\left(y^{2}\mathcal{S}+2\Gamma_{1}^{2}\partial_{\tx{d}}\ti{\mathcal{S}}\right)\partial^{2}_{x}\gamma+4y\Gamma_{1}\Gamma_{2}\partial_{\tx{d}}\ti{\mathcal{S}}\partial^{2}_{xy}\gamma+8\Gamma_{1}\Gamma_{2}\partial_{\tx{d}}\ti{\mathcal{S}}\partial_{x}\gamma
\end{eqnarray}
and
\begin{eqnarray}\label{0.8.2}
2y\partial_{y}(\mathcal{S}\Gamma_{2})&=&2y\left(\mathcal{S}\partial_{y}\Gamma_{2}+\Gamma_{2}\left(\partial_{t}\ti{\mathcal{S}}+\partial_{\tx{d}}\ti{\mathcal{S}}\partial_{y}\Gamma_{0}\right)\right)\nonumber \\
&=&y\left(\mathcal{S}\left(3\partial_{y}\gamma+y\partial^{2}_{y}\gamma\right)+2\Gamma_{2}\partial_{t}\ti{\mathcal{S}}\right)\nonumber \\
&&+2y\Gamma_{2}\partial_{\tx{d}}\ti{\mathcal{S}}\left(2\partial^{2}_{xy}\gamma+2y(\partial_{x}\gamma)^{2}+2y^{2}\partial_{x}\gamma\partial^{2}_{xy}\gamma+4\Gamma_{2}\left(3\partial_{y}\gamma+y\partial^{2}_{y}\gamma\right)\right)\nonumber \\
&=&\left(y^{2}\mathcal{S}+8y^{2}\Gamma_{2}^{2}\partial_{\tx{d}}\ti{\mathcal{S}}\right)\partial^{2}_{y}\gamma+4y\Gamma_{1}\Gamma_{2}\partial_{\tx{d}}\ti{\mathcal{S}}\partial^{2}_{xy}\gamma\nonumber \\
&&+\left(3y\mathcal{S}+24y\Gamma_{2}^{2}\partial_{\tx{d}}\ti{\mathcal{S}}\right)\partial_{y}\gamma+4y^{2}\Gamma_{2}\partial_{\tx{d}}\ti{\mathcal{S}}(\partial_{x}\gamma)^{2}+2y\Gamma_{2}\partial_{t}\ti{\mathcal{S}}.
\end{eqnarray}
Plugging the content of the two previous displays into \eqref{0.8} we obtain
\eqn{0.13}
$$
\mathcal{R}[\gamma]=\mathcal{c}_{0}\partial^{2}_{x}\gamma+2\mathcal{c}_{1}\partial^{2}_{xy}\gamma+\mathcal{c}_{2}\partial^{2}_{y}\gamma+\mathcal{L}_{0},
$$
where we set
\eqn{coeff}
$$
\begin{array}{c}
\displaystyle
\mathcal{c}_{0}:=y^{2}\mathcal{S}+2\Gamma_{1}^{2}\partial_{\tx{d}}\ti{\mathcal{S}},\qquad \quad \mathcal{c}_{1}:=4y\Gamma_{1}\Gamma_{2}\partial_{\tx{d}}\ti{\mathcal{S}},\qquad \quad \mathcal{c}_{2}:=y^{2}\left(\mathcal{S}+8\Gamma_{2}^{2}\partial_{\tx{d}}\ti{\mathcal{S}}\right)\vspace{1.5mm}\\
\displaystyle
\mathcal{L}_{0}:=\left(3y\mathcal{S}+24y\Gamma_{2}^{2}\partial_{\tx{d}}\ti{\mathcal{S}}\right)\partial_{y}\gamma+4y^{2}\Gamma_{2}\partial_{\tx{d}}\ti{\mathcal{S}}(\partial_{x}\gamma)^{2}+8\Gamma_{1}\Gamma_{2}\partial_{\tx{d}}\ti{\mathcal{S}}\partial_{x}\gamma+2y\Gamma_{2}\partial_{t}\ti{\mathcal{S}}+2\left(1-\frac{2}{p}\right)\mathcal{S}\Gamma_{2}.
\end{array}
$$
\subsection{The auxiliary equation}\label{ae.s} In view of \eqref{0.29}, we seek $\gamma$ satisfying $\mathcal{R}[\gamma]=0$. Using \eqref{0.13}-\eqref{coeff}, we show that a suitable segment of $\{x=0\}$ is noncharacteristic, rewrite $\mathcal{R}[\gamma]=0$ in Cauchy-Kovalewskaya normal form for $\partial^{2}_{x}\gamma$ and solve the resulting Cauchy problem. We set $\rr_{q}:=2^{-10}(q-1)q^{-1}$ and introduce the open sets 
\eqn{0.9}
$$
\begin{cases}
\displaystyle
\ \mathcal{Q}_{\rr_{q}}:=\{(x,y)\in \mathbb{R}^{2}\colon \snr{x}<\rr_{q}, \ \snr{y}<\rr_{q}\}\vspace{1.5mm}\\
\displaystyle
\ \mathcal{U}_{q}:=\left\{(x,y,z)\in \mathcal{Q}_{\rr_{q}}\times \mathbb{R}^{5}\colon\snr{z_{i}}<\rr_{q},\ i\in \{1,3\}, \ \snr{z_{2}+1}<\rr_{q}, \ \snr{z_{i}}<1, \ i\in \{4,5\}\right\},
\end{cases}$$ 
and, formally identifying $\mathbb{R}^{2}\times \mathbb{R}^{5}\ni(x,y,z)\equiv (x,y,(\gamma,\partial_{x}\gamma,\partial_{y}\gamma,\partial^{2}_{xy}\gamma,\partial^{2}_{y}\gamma))$, we see that
$$
\Gamma_{1}=1+y^{2}z_{2},\qquad\quad \Gamma_{2}:=z_{1}+\frac{yz_{3}}{2},\qquad\quad \Gamma_{0}=2z_{2}+y^{2}z_{2}^{2}+4\left(z_{1}+\frac{yz_{3}}{2}\right)^{2}
$$
The restrictions in \eqref{0.9} guarantee that
$$
\snr{\Gamma_{0}+2}\le2\snr{z_{2}+1}+y^{2}z_{2}^{2}+4\left(z_{1}+\frac{yz_{3}}{2}\right)^{2}\le 2\rr_{q}+\rr_{q}^{2}(1+\rr_{q})^{2}+4\rr_{q}^{2}(1+\rr_{q})^{2}<4\rr_{q},
$$
therefore 
\eqn{0.90}
$$(y,\Gamma_{0})\in\mathcal{V}_{q}:=\{(t,\tx{d})\in \mathbb{R}^{2}\colon \snr{t}<\rr_{q}, \ \snr{\tx{d}+2}<4\rr_{q}\}.$$ Moreover, if $(t,\tx{d})\in \mathcal{V}_{q}$ and $\snr{t}<\rr_{q}$, it is $1<-\tx{d}<4$, $\snr{t^{2}\tx{d}}<4\rr_{q}^{2}$, thus
\eqn{0.10}
$$
\min\left\{1+t^{2}\tx{d},(1+t^{2}\tx{d})^{(q-2)/2}\right\}>\frac{1}{2},\qquad \quad 4>\ti{\mathcal{C}}(t,\tx{d})>\frac{1}{4},\qquad \quad \ti{\mathcal{S}}(t,\tx{d})\ge \frac{1}{2^{3}},
$$
$\Gamma_{1}\ge 1/2$ on $\mathcal{U}_{q}$, and $\ti{\mathcal{S}}$ in \eqref{0.3}, and its partial derivatives $\partial_{t}\ti{\mathcal{S}}$, $\partial_{\tx{d}}\ti{\mathcal{S}}$ are analytic in $\mathcal{V}_{q}$. In particular, as $\partial_{\tx{d}}\ti{\mathcal{C}}(t,\tx{d})=-2^{-1}q(1+t^{2}\tx{d})^{-1+q/2}$ and \eqref{0.10} holds, we have
\begin{eqnarray}\label{0.11}
\frac{\partial_{\tx{d}}\ti{\mathcal{S}}(t,\tx{d})}{\ti{\mathcal{S}}(t,\tx{d})}&=&\left(\frac{q-2}{2}\right)\frac{t^{2}}{1+t^{2}\tx{d}}+\frac{q(1+t^{2}\tx{d})^{-1+q/2}}{2p\ti{\mathcal{C}}(t,\tx{d})}\nonumber \\
&\ge&\frac{q-1}{2^{4}}-(2-q)\rr_{q}\ge\frac{q-1}{2^{5}},
\end{eqnarray}
for all $(t,\tx{d})\in \mathcal{V}_{q}$. Combining \eqref{0.10}-\eqref{0.11}, we obtain
\eqn{0.12}
$$
\mathcal{c}_{0}=\mathcal{S}\left(y^{2}+\frac{2\Gamma_{1}^{2}\partial_{\tx{d}}\ti{\mathcal{S}}}{\mathcal{S}}\right)\ge \mathcal{S}\left(y^{2}+\frac{(q-1)\Gamma_{1}^{2}}{2^{4}}\right)\ge \frac{q-1}{2^{10}}\qquad \mbox{on} \ \ \mathcal{U}_{q},
$$
and $\{x=0\}\cap \mathcal{Q}_{\rr_{q}}$ is noncharacteristic. Keeping in mind \eqref{0.8.1}-\eqref{0.8.2}, our goal is to solve the auxiliary equation 
\eqn{0.30}
$$
\mathcal{R}[\gamma](x,y)=0\qquad \mbox{in a neighborhood of} \ \ (0,0).
$$
Although the divergence identity leading to $\mathcal{R}[\gamma]$ was initially derived for $y\not =0$, the differential expression $\mathcal{R}[\gamma]$ itself contains no negative power of $\snr{y}$, and all quantities in \eqref{0.13} extend to $\{y=0\}$ on $\mathcal{U}_{q}$. In fact, by \eqref{0.90} we see that on $\mathcal{U}_{q}$ it is $\Gamma_{0}<-1$, so $-q\Gamma_{0}/2>0$. We also have 
$$
\partial_{t}\ti{\mathcal{S}}(0,\tx{d})=0,\qquad \quad \partial_{\tx{d}}\ti{\mathcal{S}}(0,\tx{d})\stackrel{\eqref{0.11}}{=}\frac{q}{2p}\left(-\frac{q\tx{d}}{2}\right)^{-1-\frac{1}{p}},
$$
for all $(0,\tx{d})\in \mathcal{V}_{q}$, where the first identity above is due to the evenness in $t$ of $t\mapsto \ti{\mathcal{S}}(t,\cdot)$. Moreover, 
$$
\begin{array}{c}
\displaystyle
\mathcal{C}(x,0)=-\frac{q\Gamma_{0}(x,0)}{2},\qquad \quad \mathcal{S}(x,0)=\left(-\frac{q\Gamma_{0}(x,0)}{2}\right)^{-\frac{1}{p}},\vspace{1.5mm}\\
\displaystyle \partial_{\tx{d}}\ti{\mathcal{S}}(0,\Gamma_{0}(x,0))= \frac{q}{2p}\left(-\frac{q\Gamma_{0}(x,0)}{2}\right)^{-1-1/p}.
\end{array}
$$
Finally, $\Gamma_{1}(x,0)=1$ and $\Gamma_{2}(x,0)=\gamma(x,0)$, therefore all coefficients in \eqref{coeff} extend to $\{y=0\}$ as 
$$
\begin{array}{c}
\displaystyle
\left.\mathcal{c}_{0}\right|_{\{y=0\}}=\frac{q}{p}\left(-\frac{q\Gamma_{0}(x,0)}{2}\right)^{-1-\frac{1}{p}},\qquad \quad \left.\mathcal{c}_{1}\right|_{\{y=0\}}=0,\qquad \quad \left.\mathcal{c}_{2}\right|_{\{y=0\}}=0\vspace{1.5mm}\\
\displaystyle
\left.\mathcal{L}_{0}\right|_{\{y=0\}}=\left(\frac{4q\gamma(x,0)\partial_{x}\gamma(x,0)}{p}\right)\left(-\frac{q\Gamma_{0}(x,0)}{2}\right)^{-1-\frac{1}{p}}+2\gamma(x,0)\left(1-\frac{2}{p}\right)\left(-\frac{q\Gamma_{0}(x,0)}{2}\right)^{-\frac{1}{p}}.
\end{array}
$$
Consequently,
$$
\mathcal{R}[\gamma](x,0)=\mathcal{c}_{0}\partial^{2}_{x}\gamma(x,0)+\left.\mathcal{L}_{0}\right|_{\{y=0\}}
$$
is the analytic extension of $\mathcal{R}[\gamma]$ through $\{y=0\}$. Back to \eqref{0.30}, under the identification
$\mathbb{R}^{2}\times\mathbb{R}^{5}\ni(x,y,z)
\equiv(x,y,\gamma,\partial_{x}\gamma,\partial_{y}\gamma,
\partial_{xy}^{2}\gamma,\partial_{y}^{2}\gamma),$ the inclusion $(y,\Gamma_{0})\in\mathcal V_{q}$ on $\mathcal{U}_{q}$ together with the analyticity of $\ti{\mathcal{S}}, \partial_{t}\ti{\mathcal{S}}, \partial_{\tx{d}}\ti{\mathcal{S}}$ on $\mathcal{V}_{q}$ shows that all the coefficients in \eqref{0.13} are real analytic functions of $(x,y,z)$ on $\mathcal{U}_{q}$. Via the lower bound \eqref{0.12}, we can solve \eqref{0.13} for $\partial_{x}^{2}\gamma$. In fact, with the coefficients in \eqref{0.13} understood as functions of $(x,y,z)$, define the real analytic map
$$
\mathcal{U}_{q}\ni (x,y,z)\mapsto\mathcal{G}(x,y,z):=-\frac{2\mathcal{c}_{1}z_{4}+\mathcal{c}_{2}z_{5}+\mathcal{L}_{0}}{\mathcal{c}_{0}},
$$
and \eqref{0.30} becomes 
\eqn{0.48}
$$\partial^{2}_{x}\gamma(x,y)=\mathcal{G}(x,y,\gamma,\partial_{x}\gamma,\partial_{y}\gamma,
\partial_{xy}^{2}\gamma,\partial_{y}^{2}\gamma).$$ Notice that every derivative occurring on the right-hand side has $x$-order strictly smaller than $2$; in particular, the dependence on $\partial_{xy}^{2}\gamma$ and $\partial_{y}^{2}\gamma$ matches the assumptions of Cauchy-Kowalevskaya theorem, \cite[Chapter 1]{dib09}. We couple \eqref{0.48} with analytic Cauchy data $\gamma(0,y)=0$, $\partial_{x}\gamma(0,y)=-1$ for $\snr{y}<\rr_{q}$, satisfying $\partial_{y}\gamma(0,y)=\partial^{2}_{y}\gamma(0,y)=0=\partial^{2}_{xy}\gamma(0,y)=0$ so that, recalling also \eqref{0.9}, $(0,y,0,-1,0,0,0)\in \mathcal{U}_{q}$ for all $\snr{y}<\rr_{q}$. We then consider problem
\eqn{0.14}
$$
\begin{cases}
\displaystyle
\ \partial^{2}_{x}\gamma(x,y)=\mathcal{G}(x,y,\gamma,\partial_{x}\gamma,\partial_{y}\gamma,\partial^{2}_{xy}\gamma,\partial^{2}_{y}\gamma)\vspace{1.5mm}\\
\displaystyle
\ \gamma(0,y)=0,\qquad \quad \partial_{x}\gamma(0,y)=-1.
\end{cases}
$$
Cauchy-Kowalevskaya theorem \cite[Chapter 1]{dib09} applies and yields an open neighborhood $\mathcal{Q}_{0}\Subset \mathcal{Q}_{\rr_{q}}$ of $(0,0)$ and a real-analytic solution $\gamma$ of \eqref{0.14} on $\mathcal{Q}_{0}$ uniquely determined by the prescribed Cauchy data. Up to shrink it, we can assume that $\mathcal{Q}_{0}$ is symmetric in $y$, in the sense that $(x,y)\in \mathcal{Q}_{0}$ iff $(x,-y)\in \mathcal{Q}_{0}$. Observe that $\gamma$ is even in the $y$-variable. In fact, with $\bar{\gamma}(x,y):=\gamma(x,-y)$, via \eqref{0.13} we have that $\mathcal{R}[\bar{\gamma}](x,y)=\mathcal{R}[\gamma](x,-y)$. Moreover, $\bar{\gamma}$ match the Cauchy data in $\eqref{0.14}_{2}$, so uniqueness in Cauchy-Kowalevskaya theorem grants that $\bar{\gamma}=\gamma$. Introduce now the continuous map
$$
\mathcal{Q}_{0}\ni (x,y)\mapsto\Phi_{\gamma}(x,y):=\left(x,y,\gamma(x,y),\partial_{x}\gamma(x,y),\partial_{y}\gamma(x,y),
\partial_{xy}^{2}\gamma(x,y),\partial_{y}^{2}\gamma(x,y)\right),
$$
and observe that $\eqref{0.14}_{2}$ implies that $\Phi_{\gamma}(0,0)\in \mathcal{U}_{q}$. Since $\mathcal{U}_{q}$ is open and $\Phi_{\gamma}$ is continuous, we can find a positive number $\ell_{p}\equiv \ell_{p}(p)>0$ and cube $Q_{\ell_{p}}\equiv Q_{p}:=(-\ell_{p},\ell_{p})^{2}\Subset \mathcal{Q}_{0}$ such that
\eqn{0.17.1}
$$\Phi_{\gamma}(\bar{Q}_{p})\Subset \mathcal{U}_{q},$$ therefore $\gamma$ solves \eqref{0.30} in $Q_{p}$. Up to reduce the size of $\ell_{p}$, by \eqref{0.15}, \eqref{0.10} and $\eqref{0.14}_{2}$, it holds
\eqn{0.17}
$$
\begin{cases}
    \displaystyle
    \ \bar{Q}_{p}\subset \{(x,y)\in \mathcal{Q}_{0}\colon\Gamma_{0}(x,y)<-1\}\vspace{1.5mm}\\
    \displaystyle
    \ \frac{1}{c_{p}}\le \mathcal{C}(x,y)\le c_{p}\quad \mbox{for all} \ \ (x,y)\in \bar{Q}_{p},
\end{cases}
$$
for some $c_{p}\equiv c_{p}(p)>1$. We then look back at \eqref{0.16}: we have $Du(x,0)=(1,0)$ for all $(x,0)\in \bar{Q}_{p}$ and, by \eqref{0.4} and \eqref{0.17} also that $\snr{Du(x,y)}^{2}=1+y^{2}\Gamma_{0}(x,y)<1$ whenever $y\not =0$. This and the discussion in Section \ref{dbi.s} imply that $[-\ell_{p},\ell_{p}]\times\{0\}$ is the only maximally extended light ray for $u$, that is $\eqref{0.37}_{2}$. 
\subsection{Maximality of $u$} We analyze the stress near the light ray and verify the weak Euler equation across it. We first compute the two components of the stress field and show that the tangential component has an integrable power singularity, while the normal component tends to zero. We then use this information to prove that the stress remains divergence-free in the sense of distributions across the light segment and no concentration occurs. As a result, function $u$ in \eqref{0.18} is a local maximizer of integral $\mathcal{J}_{q}$. By \eqref{0.17}, $\mathcal{C}(x,0)$ is strictly positive on $[-\ell_{p},\ell_{p}]$, and, by \eqref{0.7} it is $\Gamma_{1}(x,0)=1$, $\Gamma_{2}(x,0)=\gamma(x,0)$. We then look back at \eqref{0.26} and, keeping in mind that $\Gamma_{1}$, $\Gamma_{2}$, $\mathcal{S}$ are analytic in a neighborhood of $\bar{Q}_{p}$, and that $\gamma$, $\Gamma_{i}$, $i\in \{0,1,2\}$, $\mathcal{C}$, $\mathcal{S}$ are even in $y$, we expand
$$
\begin{cases}
    \displaystyle
    \ \mathcal{S}(x,y)\Gamma_{1}(x,y)=\mathcal{S}(x,0)\Gamma_{1}(x,0)+\tx{O}(y^{2})=\left(-\frac{q\Gamma_{0}(x,0)}{2}\right)^{-1/p}+\tx{O}(y^{2})\vspace{1.5mm}\\
    \displaystyle
    \ \mathcal{S}(x,y)\Gamma_{2}(x,y)=\mathcal{S}(x,0)\Gamma_{2}(x,0)+\tx{O}(y^{2})=\gamma(x,0)\left(-\frac{q\Gamma_{0}(x,0)}{2}\right)^{-1/p}+\tx{O}(y^{2}),
\end{cases}
$$
uniformly in $x$, so that \eqref{0.26} reads as
\eqn{0.38}
$$
\partial\ti{J}_{q}(Du(x,y)):=\left(\ti{\tx{J}}^{1}_{q}(x,y),\ti{\tx{J}}^{2}_{q}(x,y)\right),
$$
for a.e. $(x,y)\in Q_{p}$, where it is
\eqn{0.38.1}
$$
\begin{cases}
    \displaystyle
    \ \ti{\tx{J}}^{1}_{q}(x,y)=\left(-\frac{q\snr{y}^{2}\Gamma_{0}(x,0)}{2}\right)^{-\frac{1}{p}}+\tx{O}(\snr{y}^{2-\frac{2}{p}})\vspace{1.5mm}\\
    \displaystyle
    \ \ti{\tx{J}}^{2}_{q}(x,y)=2\gamma(x,0)\textnormal{sgn}(y)\snr{y}^{1-\frac{2}{p}}\left(-\frac{q\Gamma_{0}(x,0)}{2}\right)^{-1/p}+\tx{O}(\snr{y}^{3-\frac{2}{p}}).
\end{cases}
$$
We then bound,
\eqn{0.27}
$$
\snr{\ti{\tx{J}}^{1}_{q}(x,y)}\stackrel{\eqref{0.17.1},\eqref{0.17}}{\le}\frac{c(q)}{\snr{y}^{2/p}},\qquad \quad \snr{\ti{\tx{J}}^{2}_{q}(x,y)}\stackrel{\eqref{0.17.1},\eqref{0.17}}{\le}c(q)\snr{y}^{1-2/p},
$$
for all $(x,y)\in Q_{p}$, $\snr{y}\not =0$. More precisely, as $p>2$, $\snr{\ti{\tx{J}}^{2}_{q}(x,y)}\to 0$ as $\snr{y}\to 0$ uniformly in $x\in [-\ell_{p},\ell_{p}]$, and $\partial\ti{J}_{q}(Du)\in L^{\frac{p}{2},\infty}(Q_{p},\mathbb{R}^{2})$. In fact, for every $\lambda>0$ we have
$$
Q_{p}\cap\left\{\snr{\partial\ti{J}_{q}(Du)}>\lambda\right\}\stackrel{\eqref{0.27}}{\subseteq}Q_{p}\cap\{c\snr{y}^{-\frac{2}{p}}>\lambda\},
$$
for $c\equiv c(q)$, so
\eqn{0.28}
$$
\lambda^{\frac{p}{2}}\left|Q_{p}\cap\left\{\snr{\partial\ti{J}_{q}(Du)}>\lambda\right\}\right|\le \lambda^{\frac{p}{2}}\left|Q_{p}\cap\{\snr{y}<(c/\lambda)^{\frac{p}{2}}\}\right|\le c(q),
$$
and $\eqref{0.37}_{1}$ is proven. Finally, we show that $\diver(\partial\ti{J}_{q}(Du))$ does not concentrate across the light segment. Fix $\varphi\in C^{\infty}_{c}(Q_{p})$, $0<\delta<\ell_{p}$, set $Q_{p;\delta}:=(-\ell_{p},\ell_{p})\times (\delta,\ell_{p})\cup (-\ell_{p},\ell_{p})\times (-\ell_{p},-\delta)$, observe that
\eqn{0.31}
$$
\snr{Q_{p}\setminus Q_{p;\delta}}=4\ell_{p}\delta\to_{\delta\to 0}0,
$$
and split
\begin{eqnarray*}
\int_{Q_{p}}\langle\partial\ti{J}_{q}(Du),D\varphi\rangle\dx\dy&=&\int_{Q_{p;\delta}}\langle\partial\ti{J}_{q}(Du),D\varphi\rangle\dx\dy\nonumber \\
&&+\int_{Q_{p}\setminus Q_{p;\delta}}\langle\partial\ti{J}_{q}(Du),D\varphi\rangle\dx\dy=:\mbox{(I)}_{\delta}+\mbox{(II)}_{\delta}.
\end{eqnarray*}
As on $Q_{p;\delta}$ it is $\snr{y}\not =0$, \eqref{0.29} and \eqref{0.30} yield that $\diver(\partial\ti{J}_{q}(Du))=0$, so integrating by parts in $\mbox{(I)}_{\delta}$ we have
\begin{eqnarray*}
\snr{(\mbox{I})_{\delta}}&\le&\left|\int_{Q_{p;\delta}}\diver(\partial\ti{J}_{q}(Du))\varphi\dx\dy\right|+\left|\int_{-\ell_{p}}^{\ell_{p}}\ti{\tx{J}}^{2}_{q}(x,\delta)\varphi(x,\delta)\dx\right|\nonumber \\
&&+\left|\int_{-\ell_{p}}^{\ell_{p}}\ti{\tx{J}}^{2}_{q}(x,-\delta)\varphi(x,-\delta)\dx\right|\stackrel{\eqref{0.27}}{\le}c(q)\nr{\varphi}_{L^{\infty}(Q_{p})}\delta^{1-\frac{2}{p}}\stackrel{p>2}{\to_{\delta\to 0}}0,
\end{eqnarray*}
while, concerning $(\mbox{II})_{\delta}$, it is
$$
\snr{\mbox{(II)}_{\delta}}\le c(q)\nr{D\varphi}_{L^{\infty}(Q_{p})}\snr{Q_{p}\setminus Q_{p;\delta}}^{1-\frac{2}{p}}[\partial\ti{J}_{q}(Du)]_{\frac{p}{2},\infty;Q_{p}}\stackrel{\eqref{0.28},\eqref{0.31}}{\to_{\delta\to 0}}0,
$$
therefore we can conclude that
\eqn{0.32}
$$
\int_{Q_{p}}\langle\partial\ti{J}_{q}(Du),D\varphi\rangle\dx\dy=0\qquad \mbox{for all} \ \ \varphi\in C^{\infty}_{c}(Q_{p})
$$
and $\diver(\partial \ti{J}_{q}(Du))=0$ on $Q_{p}$ in the sense of distributions. By \eqref{0.28} and the weak* density of smooth, compactly supported maps, we can extend the validity of \eqref{0.32} to all $\varphi\in W^{1,\infty}(Q_{p})\cap W^{1,1}_{0}(Q_{p})$, that is \eqref{0.36}. This and the convexity of $\ti{J}_{q}$ imply that $\ti{\mathcal{J}}_{q}(u;Q_{p})\le \ti{\mathcal{J}}_{q}(w;Q_{p})$ for all $w\in \mathcal{D}_{\infty}(u;Q_{p})$, cf. \eqref{0.33}, and the conclusion follows from \eqref{0.2} and the definitions in \eqref{0.34}-\eqref{0.35}.\\\\
\noindent All in all, we have just proven the following theorem.
\begin{theorem}\label{q.t}
Let $q\in (1,2)$ be any number. There exist $\ell_{p}\equiv \ell_{p}(q)>0$, a cube $Q_{p}:=(-\ell_{p},\ell_{p})^{2}\subset \mathbb{R}^{2}$ and a function $u\in W^{1,\infty}(Q_{p})$ with $\sup_{(x,y)\in \bar{Q}_{p}}\snr{Du(x,y)}\le 1$, analytic in a neighborhood of $Q_{p}$, such that $u$ is a local maximizer of functional $\mathcal{J}_{q}(\ \cdot\ ;Q_{p})$ in \eqref{0.0}. Specifically, integral identity
\eqn{0.36}
$$
\int_{Q_{p}}\langle\partial J_{q}(Du),D\varphi\rangle\dx\dy=0\qquad \mbox{for all} \ \ \varphi\in W^{1,1}_{0}(Q_{p})\cap W^{1,\infty}(Q_{p}),
$$
holds true,
\eqn{0.37}
$$
\begin{cases}
\displaystyle
\ \partial J_{q}(Du)\in L^{\frac{q}{2(q-1)},\infty}(Q_{p},\mathbb{R}^{2})\vspace{1.5mm}\\
\displaystyle
\ [-\ell_{p},\ell_{p}]\times \{0\}\mbox{ is the only maximally extended light ray of }u,
\end{cases}
$$
that is $u(x,0)=x$ for all $x\in [-\ell_{p},\ell_{p}]$ and $\snr{Du(x,y)}<1$ for all $(x,y)\in Q_{p}\setminus \{y=0\}$.
\end{theorem}
\begin{remark}
\emph{The proof of Theorem \ref{q.t} has been Lean-formalized \href{https://github.com/seabiscs/def26_thm3.1_lean}{here}.}
\end{remark}
\section{Non-Lipschitz minima of the degenerate area functional}\label{nl.m} 
\noindent In this section we build on the local maximizer of integral $\mathcal{J}_{q}$ from Theorem \ref{q.t} to construct a non-Lipschitz local minimizer of functional $\mathcal{A}_{p}$ in \eqref{0.1}, and show that this function is also the unique generalized minimizer of $\mathcal{A}_{p}$ with respect to its own boundary conditions. The same assumptions and notation as in Section \ref{lr.s} will be in force here.
\subsection{Construction of the minimizer}\label{dm.s} We define our prospective minimizer by integrating the singular component of the stress field from Theorem \ref{q.t} in the $y$-variable and use the integrability of the singularity to prove that the resulting function is finite and continuous across the light line. We then calculate both weak derivatives and show that its gradient is exactly the quarter-turn of the Born-Infeld stress, which is divergence-free in the sense of distributions. In the light of \eqref{0.36}-$\eqref{0.37}_{1}$, we deduce that our newly constructed function is a minimizer of $\mathcal{A}_{p}$ with its prescribed trace and that it belongs to $W^{1;\frac{p}{2},\infty}(Q_{p})$. The strict convexity of $A_{p}$, cf. $\eqref{0.94}_{1}$, then gives uniqueness. With reference to \eqref{0.38}, we set
\eqn{0.39}
$$
\bar{Q}_{p}\ni (x,y)\mapsto \ti{v}(x,y):=\int_{0}^{y}\ti{\tx{J}}^{1}_{q}(x,\tau)\d\tau,
$$
so that
$$
\snr{\ti{v}(x,y)}\stackrel{\eqref{0.27}}{\le}c\int_{0}^{\snr{y}}\snr{\tau}^{-2/p}\d\tau\le c(p)\snr{y}^{1-\frac{2}{p}},
$$
and $\ti{v}\in C(\bar{Q}_{p})$ with $\ti{v}(x,0)=0$ for all $x\in \bar{Q}_{p}\cap \{y=0\}$. By the fundamental theorem of calculus in the $y$-variable, we have $\partial_{y}\ti{v}(x,y)=\ti{\tx{J}}^{1}_{q}(x,y)$. Moreover, as $u$ is analytic in a neighborhood of $Q_{p}$ - keep in mind that $\partial\ti{J}_{q}(Du)$ only displays a singular behavior in $y$ - via \eqref{0.26} and \eqref{0.27} it is 
\eqn{0.41}
$$
\snr{\partial_{x}\ti{\tx{J}}_{q}^{1}(x,y)}\le c(p)\snr{y}^{-2/p}\qquad \mbox{for a.e.} \ \ (x,y)\in Q_{p}\setminus \{y=0\},
$$
that is integrable, so we can differentiate under the integral in \eqref{0.39} to gain $\partial_{x}\ti{v}(x,y)=\int_{0}^{y}\partial_{x}\ti{\tx{J}}^{1}_{q}(x,\tau)\d\tau$. Recall also that, by \eqref{0.29} and \eqref{0.30}, it is $
\diver(\partial\ti{J}_{q}(Du))=0$ on $Q_{p}\setminus \{y=0\}$, therefore
\eqn{0.40}
$$
\partial_{x}\ti{\tx{J}}^{1}_{q}(x,y)=-\partial_{y}\ti{\tx{J}}_{q}^{2}(x,y)\qquad \mbox{for all} \ \ (x,y)\in Q_{p}\setminus \{y=0\}.
$$
By \eqref{0.41} we can then compute, for $y>0$,
$$
\partial_{x}\ti{v}(x,y)=\lim_{\delta\to 0^{+}}\int_{\delta}^{y}\partial_{x}\ti{\tx{J}}^{1}_{q}(x,\tau)\d\tau\stackrel{\eqref{0.40}}{=}-\lim_{\delta\to 0^{+}}\int_{\delta}^{y}\partial_{\tau}\ti{\tx{J}}^{2}_{q}(x,\tau)\d\tau\stackrel{\eqref{0.27}}{=}-\ti{\tx{J}}^{2}_{q}(x,y),
$$
and similarly, if $y<0$,
$$
\partial_{x}\ti{v}(x,y)=-\lim_{\delta\to 0^{-}}\int_{y}^{\delta}\partial_{x}\ti{\tx{J}}^{1}_{q}(x,\tau)\d\tau=\lim_{\delta\to 0^{-}}\int_{y}^{\delta}\partial_{\tau}\ti{\tx{J}}^{2}_{q}(x,\tau)\d\tau=-\ti{\tx{J}}^{2}_{q}(x,y).
$$
To summarize,
\eqn{0.42}
$$
\begin{array}{c}
\displaystyle
\partial_{x}\ti{v}(x,y)=-\ti{\tx{J}}^{2}_{q}(x,y),\qquad \quad \partial_{y}\ti{v}(x,y)=\ti{\tx{J}}^{1}_{q}(x,y)\vspace{1.5mm}\\
\displaystyle
\ \ \   D\ti{v}(x,y)=\left(-\ti{\tx{J}}^{2}_{q}(x,y),\ti{\tx{J}}^{1}_{q}(x,y)\right)=\partial\ti{J}_{q}(Du(x,y))^{\perp},
\end{array}
$$
a.e. in $Q_{p}$. In particular, we deduce that $\ti{v}\in C^{0,1-\frac{2}{p}}(\bar{Q}_{p})$. In fact, set $z_{1}:=(x_{1},y_{1})$, $z_{2}:=(x_{2},y_{2})\in \bar{Q}_{p}$ and estimate\footnote{Keep in mind the well-known equivalence valid for all vectors $y_{1},y_{2}\in \mathbb{R}^{n}$, $n\ge 2$, such that $\snr{y_{1}}+\snr{y_{2}}\not =0$, and any $\alpha<1$, that is $\int_{0}^{1}\snr{y_{2}+\tau(y_{1}-y_{2})}^{-\alpha}\d\tau\approx_{n,\alpha}(\snr{y_{2}}^{2}+\snr{y_{2}-y_{1}}^{2})^{-\alpha/2}$, cf. \cite[Lemma 4.2]{sch14}.}
\begin{eqnarray}\label{0.73}
\snr{\ti{v}(z_{1})-\ti{v}(z_{2})}&\le&\snr{\ti{v}(x_{1},y_{1})-\ti{v}(x_{1},y_{2})}+\snr{\ti{v}(x_{1},y_{2})-\ti{v}(x_{2},y_{2})}\nonumber \\
&\le&\left(\int_{0}^{1}\snr{\partial_{y}\ti{v}(x_{1},y_{2}+\tau(y_{1}-y_{2}))}\d\tau\right)\snr{y_{1}-y_{2}}\nonumber \\
&&+\left(\int_{0}^{1}\snr{\partial_{x}\ti{v}(x_{1}+\tau(x_{2}-x_{1}),y_{2})}\d\tau\right)\snr{x_{1}-x_{2}}\nonumber \\
&\stackrel{\eqref{0.27}}{\le}&c\left(\int_{0}^{1}\snr{y_{2}+\tau(y_{1}-y_{2})}^{-2/p}\d\tau\right)\snr{y_{1}-y_{2}}+c\snr{y_{2}}^{1-2/p}\snr{x_{1}-x_{2}}\nonumber \\
&\le&\frac{c\snr{y_{1}-y_{2}}}{(\snr{y_{2}}^{2}+\snr{y_{1}-y_{2}}^{2})^{1/p}}+c\snr{x_{1}-x_{2}}\le c(p)\snr{z_{1}-z_{2}}^{1-2/p}.
\end{eqnarray}
Now, as $\partial A_{p}$ is well-defined for all $z\in \mathbb{R}^{2}$ if $p>2$ and
\eqn{0.45}
$$
z\mapsto z^{\perp}:=(-z_{2},z_{1})\mbox{ is an isometry},
$$
we calculate
\begin{eqnarray}\label{0.43}
\partial A_{p}(D\ti{v})&\stackrel{\eqref{0.42}}{=}&\frac{\snr{\partial\ti{J}_{q}(Du)^{\perp}}^{p-2}\partial\ti{J}_{q}(Du)^{\perp}}{(1+\snr{\partial\ti{J}_{q}(Du)^{\perp}}^{p})^{1-1/p}}\stackrel{\eqref{0.45}}{=}\frac{\snr{\partial\ti{J}_{q}(Du)}^{p-2}\partial\ti{J}_{q}(Du)^{\perp}}{(1+\snr{\partial\ti{J}_{q}(Du)}^{p})^{1-1/p}}\nonumber \\
&=&\partial A_{p}(\partial\ti{J}_{q}(Du))^{\perp}\stackrel{\eqref{0.99}}{=}Du^{\perp},
\end{eqnarray}
a.e. in $Q_{p}$, where we also used that $\snr{\partial\ti{J}_{q}(Du)}=\snr{Du}^{q-1}(1-\snr{Du}^{q})^{-1+1/q}$. Since $u$ is analytic in a neighborhood of $Q_{p}$ we have that $Du^{\perp}\in C^{1}(\bar{Q}_{p})$ and $\partial A_{p}(D\ti{v})=Du^{\perp}$ a.e. in $Q_{p}$, and
\eqn{0.44}
$$
\diver(\partial A_{p}(D\ti{v}))\stackrel{\eqref{0.43}}{=}\diver(Du^{\perp})=-\partial^{2}_{xy}u+\partial^{2}_{yx}u=0\qquad \mbox{in} \ \ Q_{p}.
$$
The convexity of $z\mapsto A_{p}(z)$, \eqref{0.44}, \eqref{0.42}, and $\eqref{0.37}_{1}$ imply that $\ti{v}\in C(\bar{Q}_{p})\cap W^{1;\frac{p}{2},\infty}(Q_{p})$ is a local minimizer of functional $\mathcal{A}_{p}$ in \eqref{0.1}, i.e., $\mathcal{A}_{p}(\ti{v};Q_{p})\le \mathcal{A}_{p}(w;Q_{p})$ for all $w\in \ti{v}+W^{1,1}_{0}(Q_{p})$, and in particular by $\eqref{0.94}_{1}$, $\ti{v}\in \mathcal{D}_{0}(\ti{v};Q_{p})$ is the unique solution to Dirichlet problem $\min_{w\in \mathcal{D}_{0}(\ti{v};Q_{p})}\mathcal{A}_{p}(w;Q_{p})$. Since $\ti{v}\in W^{1;\frac{p}{2},\infty}(Q_{p})$ and $p>2$, Proposition \ref{rel.p} then applies, and $\ti{v}$ is also the unique generalized minimizer of $\mathcal{A}_{p}$ in $\mathcal{D}_{0}(\ti{v};Q_{p})$.
\subsection{Failure of local Lipschitz regularity}\label{f.s} Here we exploit the quantitative blow-up of $D\ti{v}$ across the light-line to show the failure of local Lipschitz continuity of $\ti{v}$ inside $Q_{p}$. By \eqref{0.42}-\eqref{0.45} we have that 
\eqn{0.71}
$$\snr{D\ti{v}(x,y)}=\snr{\partial\ti{J}_{q}(Du(x,y))^{\perp}}=\snr{\partial\ti{J}_{q}(Du(x,y))}\stackrel{\eqref{0.6}}{=}\frac{(1-y^{2}\mathcal{C}(x,y))^{1/p}}{\snr{y}^{2/p}\mathcal{C}(x,y)^{1/p}},$$
for all $(x,y)\in Q_{p}\cap\{y\not =0\}$. As $\mathcal{C}$ is analytic around $Q_{p}$ and even in the $y$-variable, we expand
$$
\mathcal{C}(x,y)=\mathcal{C}(x,0)+\tx{O}(y^{2}),\qquad \mbox{for all} \ \ x\in [-\ell_{p},\ell_{p}]
$$
so that, via \eqref{0.17}, it is $\mathcal{C}(x,y)\approx_{q}1$ for all $(x,y)\in Q_{p}$, thus $1-y^{2}\mathcal{C}(x,y)=1+\tx{O}(y^{2})$ on $Q_{p}$ and there exists $\sigma_{p}\equiv \sigma_{p}(p)\in (0,\ell_{p})$ such that
\eqn{0.47}
$$
\snr{D\ti{v}(x,y)}\ge\frac{(1-c_{p}y^{2})^{1/p}}{c_{p}\snr{y}^{2/p}}\ge \frac{1}{2c_{p}\snr{y}^{2/p}},
$$
for all $(x,y)\in S_{p}:=Q_{p}\cap \{\snr{y}< \sigma_{p}\}$, $y\not =0$. We then pick any $z_{0}:=(x_{0},0)\in S_{p}\cap \{y=0\}$ with $\dist(z_{0},\partial Q_{p})>0$, fix $\rr\in (0,\min\{\sigma_{p},\dist(z_{0},\partial Q_{p})\})$, $\lambda_{\rr}:=\max\{1,(2/\rr)^{2/p}(2c_{p})^{-1}\}$, $t\in (0,\infty)$, $M\in (\lambda_{\rr},\infty)$ and compute
\begin{eqnarray}\label{0.60}
[D\ti{v}]_{\frac{p}{2},t;B_{\rr}(z_{0})}^{t}&=&\frac{p}{2}\int_{0}^{\infty}\left(\lambda^{p/2}\snr{\{(x,y)\in B_{\rr}(z_{0})\colon\snr{D\ti{v}(x,y)}>\lambda\}}\right)^{\frac{2t}{p}}\frac{\d\lambda}{\lambda}\nonumber \\
&\stackrel{\eqref{0.47}}{\ge}&\frac{p}{2}\int_{\lambda_{\rr}}^{M}\left(\lambda^{p/2}\snr{\{(x,y)\in B_{\rr}(z_{0})\colon\snr{y}^{-2/p}>2c_{p}\lambda\}}\right)^{\frac{2t}{p}}\frac{\d\lambda}{\lambda}\nonumber \\
&\ge&\frac{1}{c}\int_{\lambda_{\rr}}^{M}\frac{\d\lambda}{\lambda}=\frac{1}{c(p,t,\rr)}\log\left(\frac{M}{\lambda_{\rr}}\right),
\end{eqnarray}
which blows up as $M\to \infty$. As a consequence, $\ti{v}\in W^{1;\frac{p}{2},\infty}(Q_{p})\setminus W^{1;\frac{p}{2},t}_{\loc}(S_{p})$ for all $t\in (0,\infty)$, and a fortiori, $\ti{v}\in W^{1;\frac{p}{2},\infty}(Q_{p})\setminus W^{1,\infty}_{\loc}(Q_{p})$.\\\\
\noindent The outcome of Sections \ref{dm.s}-\ref{f.s} is summarize in the following theorem.
\begin{theorem}\label{t.1}
Let $p>2$ and $Q_{p}:=(-\ell_{p},\ell_{p})^{2}$ be the two-dimensional cube of side length $2\ell_{p}$, for some $\ell_{p}\equiv \ell_{p}(p)>0$. There exists a local minimizer $\ti{v}\in C^{0,1-\frac{2}{p}}(\bar{Q}_{p})\cap W^{1,1}(Q_{p})$ of functional $\mathcal{A}_{p}$ that is also a generalized minimizer of $\mathcal{A}_{p}$ with respect to its own boundary conditions such that $\ti{v}\in W^{1;\frac{p}{2},\infty}(Q_{p})\setminus W^{1,\infty}_{\loc}(Q_{p})$.
\end{theorem}
\begin{remark}
    \emph{The content of Section \ref{nl.m} has been Lean-formalized \href{https://github.com/seabiscs/sect_4_lean}{here}.}
\end{remark}
\section{Proof of Theorem \ref{t.2}}\label{pt.1}
\noindent In this section, we manipulate the example from Theorem \ref{t.1} to construct a Dirichlet problem driven by functional \eqref{0.1} defined on a smooth domain $\Omega\subset \mathbb{R}^{n}$ and with smooth boundary datum.  
\subsection{Construction of the domain}\label{dom.s} Here we construct an open, bounded domain $\Omega\subset \mathbb{R}^{n}$ with smooth boundary in arbitrary dimension $n\ge 2$ whose projection on $\mathbb{R}^{2}$ is compactly contained in $Q_{p}=(-\ell_{p},\ell_{p})^{2}\subset \mathbb{R}^{2}$, the cube from Theorem \ref{t.1}. Define the $C^{\infty}$-regular function 
\eqn{0.56}
$$
\mathcal{b}(t):=\begin{cases}
    \displaystyle
    \ \textnormal{sgn}(t)\exp\{-t^{-2}\}\quad &\mbox{if} \ \ t\not =0\vspace{1.5mm}\\
    \displaystyle
    \ 0\quad &\mbox{if} \ \ t=0,
\end{cases}
$$
identify $x\equiv(x_{1},\cdots,x_{n})\equiv (\bar{x},y)$, where $\bar{x}=x_{1}$, $y=x_{2}$ if $n=2$, or $\bar{x}=(x_{1},x_{3},\cdots,x_{n})$, $y=x_{2}$ when $n\ge  3$, introduce parameter $r:=\ell_{p}/2$, and the smooth maps
\eqn{0.49.1}
$$
\mathcal{s}(\bar{x}):=\frac{\snr{\bar{x}}^{2}-r^{2}}{r^{2}},\qquad \quad \Lambda(\bar{x},y):=\left(1+\frac{2y}{r}\right)^{2}+e\mathcal{b}(\mathcal{s}(\bar{x}))-1.
$$
Our domain is then given by 
$$\Omega:=\left\{x\equiv (\bar{x},y)\in \mathbb{R}^{n}\colon \Lambda(\bar{x},y)<0\right\}.$$
Notice that $\Omega$ is open as $\Lambda$ is continuous. Moreover, $\mathcal{s}(\bar{x})\ge -1$ and $\mathcal{s}\mapsto \mathcal{b}(\mathcal{s})$ is strictly increasing with $\mathcal{b}(\mathbb{R})\subset (-1,1)$, so $e\mathcal{b}(\mathcal{s}(\bar{x}))\ge -1$ and
$$
\Lambda(\bar{x},y)<0 \ \Longrightarrow \ 1>e\mathcal{b}(\mathcal{s}(\bar{x})) \ \Longrightarrow \ 1>\mathcal{s}(\bar{x}),
$$
so that 
\eqn{0.49}
$$
\Omega=\left\{x\equiv (\bar{x},y)\colon \snr{\bar{x}}<\sqrt{2}r, \ y\in (y_{-}(\bar{x}),y_{+}(\bar{x}))\right\},
$$
where
\eqn{0.67}
$$
y_{-}(\bar{x}):=-\frac{r}{2}\left(1+\sqrt{1-e\mathcal{b}(\mathcal{s}(\bar{x}))}\right),\qquad \quad y_{+}(\bar{x}):=\frac{r}{2}\left(\sqrt{1-e\mathcal{b}(\mathcal{s}(\bar{x}))}-1\right).
$$
Since by \eqref{0.56} we have $\sqrt{1-e\mathcal{b}(\mathcal{s}(\bar{x}))}\le\sqrt{2}$, by \eqref{0.49}-\eqref{0.67} $\Omega$ is bounded. In particular, $\{\snr{\bar{x}}<\sqrt{2}r\}\times \{-r/2\}\subset \Omega$, therefore we can show that $\Omega$ is path connected. In fact, for all $(\bar{x},y)\in \Omega$ we have $(\bar{x},-r/2)\in \{\snr{\bar{x}}<\sqrt{2}r\}\times \{-r/2\}$, so we fix any $\tau\in [0,1]$ and estimate
\begin{eqnarray*}
\Lambda\left(\bar{x},-\frac{r}{2}+\tau\left(y+\frac{r}{2}\right)\right)&=&\left(1+\frac{2}{r}\left(-\frac{r}{2}+\tau\left(y+\frac{r}{2}\right)\right)\right)^{2}+e\mathcal{b}(\mathcal{s}(\bar{x}))-1\nonumber \\
&\le&e\mathcal{b}(\mathcal{s}(\bar{x}))-1+\tau\left(1+\frac{2y}{r}\right)^{2}\nonumber \\
&\le& (1-\tau)\Lambda\left(\bar{x},-\frac{r}{2}\right)+\tau\Lambda(\bar{x},y)<0,
\end{eqnarray*}
so that every point of $\Omega$ can be joined to a point in $\{\snr{\bar{x}}<\sqrt{2}r\}\times \{-r/2\}$ with a segment. Being $\{\snr{\bar{x}}<\sqrt{2}r\}\times \{-r/2\}$ path connected, $\Omega$ is path connected as well. Now we only need to take care of the regularity of the boundary. In the light of the regular set theorem, we need to look at $D\Lambda(\bar{x},y)$ for all those points $(\bar{x},y)\in \mathbb{R}^{n}$ such that $\Lambda(\bar{x},y)=0$. Direct computations give
$$
D\Lambda(\bar{x},y)=\left(\frac{2e\mathcal{b}'(\mathcal{s}(\bar{x}))\bar{x}}{r^{2}},\frac{4}{r}\left(1+\frac{2y}{r}\right)\right),
$$
so if $y\not =-r/2$, $D\Lambda(\bar{x},y)\not =0$ and there is nothing else to prove, so we focus on the case $y=-r/2$. We then have $\Lambda(\bar{x},-r/2)=e\mathcal{b}(\mathcal{s}(\bar{x}))-1=0$, which implies that $\mathcal{s}(\bar{x})=1$, $\snr{\bar{x}}=\sqrt{2}r$ and $D\Lambda(\bar{x},-r/2)=(4\bar{x}r^{-2},0)\not =0$. As a consequence, $\partial \Omega$ is smooth. Next if $n=2$, \eqref{0.49} and the very definition of $r$ imply that $\Omega\Subset Q_{p}$, while if $n\ge 3$, we denote by $\pi_{1;2}$ the projection on the plane spanned by variables $(x_{1},y)$ to gain
\eqn{0.59}
$$
\pi_{1;2}(\bar{\Omega})\subseteq\left[-\sqrt{2}r,\sqrt{2}r\right]\times \left[-\frac{r(1+\sqrt{2})}{2},\frac{r(\sqrt{2}-1)}{2}\right]\stackrel{r=\ell_{p}/2}{\Subset}Q_{p}.
$$
Finally
\eqn{0.65}
$$
\begin{cases}
    \displaystyle
    \ \Omega\cap \{y=0\}=\{\snr{\bar{x}}<r\}\times\{0\}\vspace{1.5mm}\\
    \displaystyle
    \ \partial\Omega\cap \{y=0\}=\{\snr{\bar{x}}=r\}\times \{0\},
\end{cases}
$$
which, specialized for $n=2$ reads as
\eqn{0.66}
$$
\begin{cases}
    \displaystyle
    \ \Omega\cap \{y=0\}=(-r,r)\times\{0\}\subset (-\ell_{p},\ell_{p})\times \{0\}\vspace{1.5mm}\\
    \displaystyle
    \ \partial\Omega\cap \{y=0\}=\{(\pm r,0)\},
\end{cases}
$$
follow from $\Lambda(\bar{x},0)=e\mathcal{b}(\mathcal{s}(\bar{x}))$, thus $\Lambda(\bar{x},0)<0$ implies via \eqref{0.56} that $\mathcal{s}(\bar{x})<0$, i.e. $\snr{\bar{x}}<r$, while $\Lambda(\bar{x},0)=0$ yields $\snr{\bar{x}}=r$, and \eqref{0.65} is proven.

\subsection{Smooth boundary traces}\label{tr.s} Here we construct a suitable $C^{\infty}$-regular boundary datum for problem \eqref{pd}. Before entering into the details, let us extend to higher dimension the minimizer $\ti{v}$ constructed in Theorem \ref{t.1}. Keeping in mind the notation fixed below \eqref{0.56}, with $x\in \bar{\Omega}$, if $n=2$ we just rename $v(x)\equiv v(\bar{x},y):=\ti{v}(x_{1},y)$ without further changes, while if $n\ge 3$, with $\pi_{1;2}$ being the $2$-d projection from \eqref{0.59} and set $v(x)\equiv v(\bar{x},y):=\ti{v}(\pi_{1;2}(x))=\ti{v}(x_{1},y)$ - a meaningful position as if $(\bar{x},y)\in \bar{\Omega}$, then $(x_{1},y)\in \pi_{1;2}(\bar{\Omega})\Subset Q_{p}$, cf. \eqref{0.59}. Notice that by definition it is 
\eqn{0.72}
$$Dv=(\partial_{x_{1}}\ti{v},\partial_{x_{2}}\ti{v},0,\cdots,0) \ \Longrightarrow \ \snr{Dv}=\snr{D\ti{v}}\qquad \mbox{a.e. in} \ \ \Omega.$$
Let us rearrange \eqref{0.39} in a more convenient way:
\begin{eqnarray}\label{0.53}
v(x)&=&\ti{v}(x_{1},y)=\int_{0}^{y}\ti{\tx{J}}^{1}_{q}(x_{1},\tau)\d\tau=\int_{0}^{1}y\ti{\tx{J}}^{1}_{q}(x_{1},\tau y)\d\tau\nonumber \\
&\stackrel{\eqref{0.26}}{=}&\textnormal{sgn}(y)\snr{y}^{1-2/p}\left(\int_{0}^{1}\tau^{-2/p}\mathcal{S}(x_{1},\tau y)\Gamma_{1}(x_{1},\tau y)\d\tau\right)=:\textnormal{sgn}(y)\snr{y}^{1-2/p}\mathcal{K}(x_{1},y),
\end{eqnarray}
so that the decay in $y$ is explicit, cf. \eqref{0.10}. By \eqref{0.65}-\eqref{0.66}, $\partial \Omega$ crosses $\{y=0\}$, i.e., the only set where $D\ti{v}$ can be unbounded, in $\{\snr{\bar{x}}=r\}\times \{0\}$ - specifically, the upper component $y_{+}$ of $\partial \Omega$ in \eqref{0.49}-\eqref{0.67} crosses $\{\snr{\bar{x}}=r\}\times\{0\}$ as it satisfies $\left.y_{+}(\bar{x})=0\right|_{\{\snr{\bar{x}}=r\}}$. As $\left.\mathcal{b}(\mathcal{s}(\bar{x}))\right|_{\{\snr{\bar{x}}=r\}}=0$, it is $\sqrt{1-e\mathcal{b}(\mathcal{s}(\bar{x}))}=1$ for all $\bar{x}\in \{\snr{\bar{x}}=r\}$ in \eqref{0.67}, therefore there is a neighborhood $N_{+}\subset \mathbb{R}^{n-1}$ of $\{\snr{\bar{x}}=r\}$ such that $y_{+}\in C^{\infty}(N_{+})$. Observe that
\eqn{0.54}
$$
\begin{cases}
\displaystyle
\ y_{+}(\bar{x})(y_{+}(\bar{x})+r)=\left(\frac{r}{2}\right)^{2}\left(\sqrt{1-e\mathcal{b}(\mathcal{s}(\bar{x}))}-1\right)\left(1+\sqrt{1-e\mathcal{b}(\mathcal{s}(\bar{x}))}\right)=-e\mathcal{b}(\mathcal{s}(\bar{x}))\left(\frac{r}{2}\right)^{2}\vspace{1.5mm}\\
\displaystyle
\ y_{+}(\bar{x})+r>0.
\end{cases}
$$
By \eqref{0.56}, $\eqref{0.49.1}$, and \eqref{0.54} it is $\textnormal{sgn}(y_{+})=-\textnormal{sgn}(\mathcal{s})$ and
\eqn{0.57}
$$
\textnormal{sgn}(y_{+})\snr{y_{+}}^{1-2/p}= -\left(\frac{er^{2}}{4}\right)^{1-\frac{2}{p}}\frac{\mathcal{b}_{p}(\mathcal{s}(\bar{x}))}{(y_{+}(\bar{x})+r)^{1-2/p}},
$$
where we set
$$
\mathcal{b}_{p}(\mathcal{s}(\bar{x})):=\begin{cases}
\displaystyle
\ \textnormal{sgn}(\mathcal{s}(\bar{x}))\exp\left\{-\left(1-\frac{2}{p}\right)\mathcal{s}(\bar{x})^{-2}\right\}\quad &\mbox{if} \ \ \mathcal{s}(\bar{x})\not =0\vspace{1.5mm}\\
\displaystyle
\ 0\quad &\mbox{if} \ \ \mathcal{s}(\bar{x})=0.\end{cases}
$$
Combining \eqref{0.53} and \eqref{0.57}, we obtain
\eqn{0.58}
$$
v(\bar{x},y_{+}(\bar{x}))=-\left(\frac{er^{2}}{4}\right)^{1-\frac{2}{p}}\frac{\mathcal{b}_{p}(\mathcal{s}(\bar{x}))\mathcal{K}(x_{1},y_{+}(\bar{x}))}{(y_{+}(\bar{x})+r)^{1-2/p}}.
$$
The factors on the right-hand side of \eqref{0.58} are smooth, cf. \eqref{0.53} and Section \ref{ss.s}, and the multiplicative factor $\mathcal{b}_{p}(\mathcal{s}(\bar{x}))$, which vanishes at $\mathcal{s}=0$ of any order, assures that all tangential derivatives vanish on the contact set $\{\snr{\bar{x}}=r\}$. On the rest of $\partial \Omega$ we have $\snr{y}>0$ so $v$ is $C^{\infty}$-regular - recall that $\mathcal{K}$ is smooth as $\mathcal{S}$, $\Gamma_{1}$ are analytic see Section \ref{ae.s}, and $\tau^{-2/p}\in L^{1}(0,1)$, therefore we can define $\bar{v}_{0}:=\left.v\right|_{\partial \Omega}\in C^{\infty}(\partial\Omega)$, and extend\footnote{Project $\bar{v}_{0}$ on a tubular neighborhood of $\partial \Omega$ and cut-off it so to construct function $v_{0}\in C^{\infty}(\mathbb{R}^{n})$ such that $\left.v_{0}\right|_{\partial\Omega}=\left.\bar{v}_{0}\right|_{\partial \Omega}$.} it in a standard way to a smooth map $v_{0}\in C^{\infty}(\mathbb{R}^{n})$ such that $\left.v_{0}\right|_{\partial\Omega}=\left.\bar{v}_{0}\right|_{\partial\Omega}.$ By Theorem \ref{t.1} it is $\ti{v}\in C^{0,1-2/p}(\bar{Q}_{p})\cap W^{1;\frac{p}{2},\infty}(Q_{p})$ so via \eqref{0.59} we get that $v\in C^{0,1-2/p}(\bar{\Omega})\cap W^{1;\frac{p}{2},\infty}(\Omega)$. In fact, the $(1-2/p)$-H\"older continuity is a direct consequence of the definition of $v$ and \eqref{0.73}. Concerning the Sobolev-Marcinkiewicz regularity, by \eqref{0.49}, \eqref{0.67}, and \eqref{0.59} we see that 
\eqn{0.74}
$$\Omega\Subset Q_{p}\times (-2r,2r)^{n-2}=:\ti{Q}_{n}$$ so for all $\lambda>0$ it is
\eqn{0.75}
$$
\lambda^{\frac{p}{2}}\left|\left\{x\in \Omega\colon \snr{Dv(x)}>\lambda\right\}\right|\stackrel{\eqref{0.72}}{\le}c\lambda^{\frac{p}{2}}\left|\left\{(x_{1},x_{2})\in Q_{p}\colon \snr{D\ti{v}(x_{1},x_{2})}>\lambda\right\}\right|\stackrel{\eqref{0.71},\eqref{0.37}}{\le}c(n,p).
$$
In particular, the Sobolev trace of $v$ agrees with $\bar{v}_{0}$, cf. \cite[Section 5.5, Theorems 1-2]{eva10}, that is
\eqn{0.70}
$$\left.v_{0}\right|_{\partial\Omega}=\left.\bar{v}_{0}\right|_{\partial\Omega}=\left.v\right|_{\partial\Omega}\qquad \mbox{and}\qquad v-v_{0}\in W^{1,1}_{0}(\Omega).$$
The smooth boundary datum is constructed.
\subsection{Non-Lipschitz solutions} Let us prove that $v$ solves Dirichlet problem \eqref{pd}. We already saw in \eqref{0.75}-\eqref{0.70} that $v\in W^{1;\frac{p}{2},\infty}(\Omega)$ attains the smooth trace $v_{0}$ on $\partial \Omega$, so we only need to show its minimality in Dirichlet class $v_{0}+W^{1,1}_{0}(\Omega)$. We have
\begin{eqnarray}\label{0.46}
\partial A_{p}(Dv)&=&\frac{\snr{Dv}^{p-2}Dv}{(1+\snr{Dv}^{p})^{1-1/p}}\nonumber \\
&\stackrel{\eqref{0.72}}{=}&\frac{\snr{D\ti{v}}^{p-2}}{(1+\snr{D\ti{v}}^{p})^{1-1/p}}\left(\partial_{x_{1}}\ti{v},\partial_{x_{2}}\ti{v},0,\cdots,0\right)=\left(\partial A_{p}(D\ti{v}),0,\cdots,0\right),
\end{eqnarray}
a.e. in $\Omega$. We take $\varphi\in W^{1,1}_{0}(\Omega)$ and extend it (without relabelling) as zero in $\ti{Q}_{n}\setminus \Omega$, see \eqref{0.74}, so that $\varphi\in W^{1,1}_{0}(\ti{Q}_{n})$. We then observe that $v+\varphi\in v_{0}+W^{1,1}_{0}(\Omega)$ and $\varphi \equiv 0$ in $\ti{Q}_{n}\setminus \Omega$, and use the convexity of $A_{p}$, \eqref{0.46}, \eqref{0.49}, \eqref{0.59}, \eqref{0.44} to bound
\begin{eqnarray*}
\mathcal{A}_{p}(v+\varphi;\Omega)-\mathcal{A}_{p}(v;\Omega)&\ge& \int_{\Omega}\langle\partial A_{p}(Dv),D\varphi\rangle\dx=\int_{\ti{Q}_{n}}\langle\partial A_{p}(D\ti{v}),D_{(x_{1},x_{2})}\varphi\rangle\dx\nonumber \\
&=&\int_{(-2r,2r)^{n-2}}\left(\int_{Q_{p}}\langle\partial A_{p}(D\ti{v}),D_{(x_{1},x_{2})}\varphi\rangle\dx_{1}\dx_{2}\right)\dx_{3}\cdots\dx_{n}=0,
\end{eqnarray*}
and $v\in \mathcal{D}_{0}(v_{0};\Omega)$ is the unique (by $\eqref{0.94}_{1}$) solution to \eqref{pd}. Proposition \ref{rel.p} then yields that $v$ is also the unique generalized minimizer of $\mathcal{A}_{p}$ in class $\mathcal{D}_{0}(v_{0};\Omega)$. We conclude by showing that $v\in W^{1;\frac{p}{2},\infty}(\Omega)\setminus W^{1,\infty}_{\loc}(\Omega)$. If $n=2$, there is nothing to prove by \eqref{0.37}, \eqref{0.71} and \eqref{0.60}. On the other hand, if $n\ge 3$ by \eqref{0.75} it is $v\in W^{1;\frac{p}{2},\infty}(\Omega)$, so we only need to show the failure of local Lipschitz regularity. Fix any point $x_{0}\equiv (\bar{x}_{0},0)\equiv (x_{0;1},0,\cdots,x_{0;n})\in \Omega\cap \{y=0\}$, radius $\rr\in (0,\dist(x_{0},\partial \Omega))$, so that $B_{\rr}(x_{0})\Subset \Omega$. With reference to Section \ref{f.s}, we set $\rr_{p}:=\min\{\rr,\sigma_{p}\}/4$ and introduce cylinder $C_{\rr_{p}}(x_{0}):=\{\bar{x}\in \mathbb{R}^{n-1}\colon \snr{\bar{x}-\bar{x}_{0}}<\rr_{p}\}\times \{-\rr_{p}<x_{2}<\rr_{p}\}\Subset B_{\rr}(x_{0})\Subset \Omega$. Notice that $\pi_{1;2}(C_{\rr_{p}}(x_{0}))\Subset \pi_{1;2}(B_{\rr}(x_{0}))\Subset Q_{p}$ and, in particular, $\snr{x_{2}}<\sigma_{p}$, so \eqref{0.47} and \eqref{0.72} yield that $\snr{Dv(x)}\ge (2c_{p})^{-1}\snr{x_{2}}^{-2/p}$ a.e. on $C_{\rr_{p}}(x_{0})$. We then have
$$
C_{\rr_{p}}(x_{0})\cap \left\{0<\snr{x_{2}}<\min\left\{\rr_{p},(2c_{p}\lambda)^{-p/2}\right\}\right\}\subseteq \left\{x\in B_{\rr}(x_{0})\colon \snr{Dv(x)}>\lambda\right\},
$$
therefore for $\lambda\ge (2c_{p}\rr_{p}^{2/p})^{-1}$ we obtain
$$
\left|\left\{x\in B_{\rr}(x_{0})\colon \snr{Dv(x)}>\lambda\right\}\right|\ge \left|C_{\rr_{p}}(x_{0})\cap \left\{0<\snr{x_{2}}<(2c_{p}\lambda)^{-p/2}\right\}\right|\approx_{n,p,\rr}\lambda^{-\frac{p}{2}}
$$
so that, proceeding as in \eqref{0.60}, we get that $v\in W^{1;\frac{p}{2},\infty}(\Omega)\setminus W^{1,\infty}_{\loc}(\Omega)$. Let us conclude with the observation that set $\Omega\cap\{y=0\}$, cf. $\eqref{0.65}_{1}$, is made of non-Lebesgue points of $Dv$. Let $n\ge 2$, $x_{0}\in \Omega\cap\{y=0\}$, $B_{\rr}(x_{0})\Subset \Omega$, recall that $Q_{\frac{\rr}{\sqrt{n}}}(x_{0})\subseteq B_{\rr}(x_{0})\subseteq Q_{\rr}(x_{0})$, and estimate
\begin{eqnarray*}
\snr{(Dv)_{B_{\rr}(x_{0})}}&\stackrel{\eqref{0.72}}{\ge}&\snr{(\partial_{x_{2}}\ti{v})_{B_{\rr}(x_{0})}}\nonumber \\
&\stackrel{\eqref{0.42},\eqref{0.38.1}}{\ge}&\left|\mint_{B_{\rr}(x_{0})}\left(-\frac{q\Gamma_{0}(x_{1},0)}{2}\right)^{-\frac{1}{p}}\frac{1}{\snr{x_{2}}^{2/p}}\dx_{1}\cdots\dx_{n}\right|-c\left|\mint_{B_{\rr}(x_{0})}\snr{x_{2}}^{2-\frac{2}{p}}\dx_{1}\cdots\dx_{n}\right|\nonumber \\
&\ge&\frac{1}{c}\left|\mint_{Q_{\frac{\rr}{\sqrt{n}}}(x_{0})}\left(-\frac{q\Gamma_{0}(x_{1},0)}{2}\right)^{-\frac{1}{p}}\frac{1}{\snr{x_{2}}^{2/p}}\dx_{1}\cdots\dx_{n}\right|\nonumber \\
&&-c\left|\mint_{Q_{\rr}(x_{0})}\snr{x_{2}}^{2-\frac{2}{p}}\dx_{1}\cdots\dx_{n}\right|\nonumber \\
&\ge&\frac{1}{c\rr}\left(\mint_{-\frac{\rr}{\sqrt{n}}}^{\ \ \frac{\rr}{\sqrt{n}}}\left(-\frac{q\Gamma_{0}(x_{0;1}-x_{1},0)}{2}\right)^{-\frac{1}{p}}\dx_{1}\right)\left(\int_{-\frac{\rr}{\sqrt{n}}}^{\ \ \frac{\rr}{\sqrt{n}}}\frac{\dx_{2}}{\snr{x_{0;2}-x_{2}}^{2/p}}\right)-c\rr^{2-\frac{2}{p}}\nonumber\\
&\ge&\frac{1}{c\rr^{2/p}}\left(\mint_{-\frac{\rr}{\sqrt{n}}}^{\ \ \frac{\rr}{\sqrt{n}}}\left(-\frac{q\Gamma_{0}(x_{0;1}-x_{1},0)}{2}\right)^{-\frac{1}{p}}\dx_{1}\right)-c\rr^{2-\frac{2}{p}},
\end{eqnarray*}
with $c\equiv c(n,p)$. Taking the $\liminf$ as $\rr\to 0$ above we achieve
\begin{eqnarray*}
\liminf_{\rr\to 0}\snr{(Dv)_{B_{\rr}(x_{0})}}&\ge&\frac{1}{c}\left(\liminf_{\rr\to 0}\rr^{-\frac{2}{p}}\right)\left(\liminf_{\rr\to0}\mint_{-\frac{\rr}{\sqrt{n}}}^{\ \ \frac{\rr}{\sqrt{n}}}\left(-\frac{q\Gamma_{0}(x_{0;1}-x_{1},0)}{2}\right)^{-\frac{1}{p}}\dx_{1}\right)\nonumber \\
&&-c\limsup_{\rr\to 0}\rr^{2-\frac{2}{p}}\nonumber \\
&\ge&\frac{1}{c}\left(-\frac{q\Gamma_{0}(x_{0;1},0)}{2}\right)^{-\frac{1}{p}}\left(\liminf_{\rr\to 0}\rr^{-\frac{2}{p}}\right)=\infty,
\end{eqnarray*}
where we also used that $-\Gamma_{0}$ is analytic and strictly positive, cf. \eqref{0.15} and \eqref{0.17}. This means that no $x_{0}\in \Omega\cap \{y=0\}$ can be a Lebesgue point of $Dv$. On the other hand, \eqref{0.42}, \eqref{0.72}, \eqref{0.26} and the analyticity of $\Gamma_{1}$, $\Gamma_{2}$, and $\mathcal{S}$ assure that $v$ is smooth in $\Omega\cap \{y\not =0\}$, therefore the set of non-Lebesgue point of $Dv$ is precisely $\Omega\cap \{y=0\}$, and $\mathcal{H}^{n-1}(\Omega\cap \{y=0\})=\mathcal{H}^{n-1}(\{\snr{\bar{x}}<r\}\times\{y=0\})\approx_{n}r^{n-1}>0$, cf. \eqref{0.65}-\eqref{0.66}. The proof is complete.
\begin{remark}\label{rmnc}
    \emph{Theorem \ref{t.2} does not contradict the global result in \cite{tau78}. Indeed, the fact that $\mathcal{b}$ vanishes of infinite order at zero implies $D_{\bar{x}}y_{+}=D^{2}_{\bar{x}}y_{+}=0$ on $\{\snr{\bar{x}}=r\}$, and so the full second fundamental form and the mean curvature of $\partial \Omega$ vanish along $\{\snr{\bar{x}}=r\}\times \{0\}$. Thus the strict mean-convexity hypothesis in \cite[Theorem 6, (iv)]{tau78} is violated. Also the non-strict, $C^{3}$-variant in \cite[Remarks, p.~168]{tau78} does not hold here, as the mean curvature of the upper boundary, with respect to the inward normal convention of \cite{tau78}, is negative at $(\bar{x},y_{+}(\bar{x}))$ for all $\snr{\bar{x}}<r$ is sufficiently close to $r$. Furthermore, $\Omega$ is not convex: at every point $(\bar{x}_{0},0)$ with $\snr{\bar{x}_{0}}=r$, the tangent hyperplane is $\{y=0\}$, while, for $(\bar{x},y_{+}(\bar{x}))$ belonging to a neighborhood of $\{\snr{\bar{x}}=r\}\times \{y=0\}$ we have $y_{+}(\bar{x})>0$ if $\snr{\bar{x}}<r$ and $y_{+}(\bar{x})<0$ for $\snr{\bar{x}}>r$, and this tangent hyperplane is not supporting.}
\end{remark}
\begin{remark}
    \emph{The content of Sections \ref{dom.s}-\ref{tr.s} has been Lean-formalized \href{https://github.com/seabiscs/def26_domain_trace_lean}{here}.}
\end{remark}

\subsection*{AI Statement} ChatGPT (OpenAI, v.~5.6) was used to double-check the proofs and Lean-formalize some arguments. Mathematical ideas belong to the author, who takes responsibility of the content of the paper.
\end{document}